\documentclass[11pt, reqno]{amsart}
\usepackage{amssymb, amsmath, amsthm, mathrsfs, amsfonts, latexsym,xcolor,euscript}

\usepackage{hyperref}
\hypersetup{
  colorlinks   = true,
  citecolor    = blue,
  linkcolor    = black
}

\usepackage[margin=1in]{geometry}

\newtheorem{theorem}{Theorem}[section]
\newtheorem{lemma}[theorem]{Lemma}

\newtheorem{proposition}[theorem]{Proposition}
\newtheorem{corollary}[theorem]{Corollary}
\newtheorem{definition}[theorem]{Definition}
\newtheorem{remark}[theorem]{Remark}

\def\dimh{{\rm dim}_{_{\rm H}}}
\def\R{{\mathbb R}}
\def\Q{{\mathbb Q}}

\def\E{{\mathbb E}}
\def\P{{\mathbb P}}

\newcommand{\op}{\operatorname}

\numberwithin{equation}{section}

\begin{document}

\title[Hitting probability, thermal capacity, and Hausdorff dimension results]{Hitting probabilities, thermal capacity, 
and Hausdorff dimension results for the Brownian sheet}

\author{Cheuk Yin Lee and Yimin Xiao}

\address[Cheuk Yin Lee]{School of Science and Engineering, The Chinese University of Hong Kong, Shenzhen, 
Guangdong 518172, China}
\email{leecheukyin@cuhk.edu.cn}

\address[Yimin Xiao]{Department of Statistics and Probability, Michigan State University, East Lansing, MI 48824, USA}
\email{xiao@stt.msu.edu}




\begin{abstract}
Let $W= \{W(t): t \in \mathbb{R}_+^N \}$ be an $(N, d)$-Brownian sheet and let $E \subset (0, \infty)^N$ and $F \subset \mathbb{R}^d$ 
be compact sets. We prove a necessary and sufficient condition for $W(E)$ to intersect $F$ with positive probability and determine the essential supremum of the Hausdorff dimension of the intersection set $W(E)\cap F$ in terms of the thermal capacity of $E \times F$. 
This extends the previous results of Khoshnevisan and Xiao (2015) for the Brownian motion and Khoshnevisan and Shi (1999) for the 
Brownian sheet in the special case when $E \subset (0, \infty)^N$ is an interval. 
\end{abstract}

\maketitle

\section{Introduction}

Hitting probabilities and fractal properties are important features of stochastic processes and random fields. Let $\{ X(t) : t \in 
\mathbb{R}^N \}$ be a random field with values in $\mathbb{R}^d$ and $E \subset \mathbb{R}^N$, $F \subset \mathbb{R}^d$ be 
arbitrary compact sets. We are interested in the following two questions.
\begin{enumerate}
\item[(1)] What is a necessary and sufficient condition for $\P(X(E) \cap F \ne \varnothing) > 0$?
\item[(2)] What is the Hausdorff dimension $\dimh (X(E) \cap F)$ on the event $X(E) \cap F \ne \varnothing$?
\end{enumerate}

For the $d$-dimensional Brownian motion $\{B(t) : t \ge 0 \}$, Question (1) was answered by Watson 
\cite{Watson:TC,Watson} and Doob \cite{Doob} by using analytic and probabilistic methods, respectively. More precisely, 
Watson \cite{Watson:TC,Watson} developed the relevant potential theory and established several different characterizations of 
the polar sets in $\R^{d+1}$ for the heat equation, and Doob \cite{Doob} showed that a set $G \subset \R^{d+1}$ is polar for the 
heat equation if and only if the Markov process $\{(B(t), t): t \ge 0\}$ hits $G$ with zero probability from every starting point.
Taylor and Watson \cite{Taylor:TW} provided a necessary condition and a sufficient condition for $\P(X(E) \cap F \ne \varnothing) > 0$ 
in terms of the Hausdorff dimension of $E \times F$ under a parabolic metric on $\R^{d+1}$.
When $E$ is an interval in $(0, \infty)^N$, Question (1) for the Brownian sheet was resolved by Khoshnevisan and Shi in their 
seminal paper \cite{KS1999}. Dalang and Nualart \cite{DN04} extended the results in \cite{KS1999} to the solutions to a class of 
nonlinear hyperbolic stochastic partial differential equations (SPDEs) in the plane. Further results about hitting probabilities for 
other Gaussian random fields or SPDEs for the special case that $E \subset (0, \infty)^N$ is an interval can be found in \cite{BLX09, 
DKN13, DP20, DP24, DS15, LSXY23, X09}. However, a complete answer to Question (1) had not been obtained for any random fields. 

For the Brownian motion, Question (2) was answered by Khoshnevisan and Xiao \cite{KX2015}, for all $d \ge 1$, in terms of 
thermal energy (see Definition \ref{thermal:cap} below) and, for $d\ge 2$, also in terms of the Hausdorff dimension of $E \times F$ 
under the parabolic metric.  As a consequence, they also obtained in \cite[Proposition 1.4]{KX2015} a criterion for $\P(B(E) \cap F 
\ne \varnothing) > 0$ that is slightly simpler than the one in  \cite{Doob, Watson:TC,Watson}.

In this paper, we extend the results in \cite{Doob,Watson:TC,Watson,KX2015} from the Brownian motion to the Brownian sheet $W= 
\{ W(t) : t \in \mathbb{R}^N_+ \}$. 
Our main results are Theorems \ref{T:hit} and \ref{T:dim:cap}, which provide answers to Questions (1) and (2) in terms of the 
thermal capacity in (\ref{def:TC}) for the Brownian sheet. A more explicit expression for the Hausdorff dimension of the intersection 
$W(E)\cap F$ is given in Theorem \ref{T:dim:dim-d}. We hope that the arguments developed in this paper 
will be useful for answering Questions (1) and (2) for other random fields, such as the solution to hyperbolic SPDEs in \cite{DN04}.


Let $W= \{ W(t) : t \in \mathbb{R}^N_+ \}$ be an $(N, d)$-Brownian sheet defined on a complete probability space $(\Omega,
 \mathscr{F}, \P)$, that is, $W$ is an $N$-parameter, $\mathbb{R}^d$-valued continuous Gaussian random field with mean zero 
 and covariance
\[ 
\E(W_i(s) W_j(t)) = \begin{cases}
\prod_{k=1}^N \min\{s_k, t_k\} & \text{ if }i = j\\
0 & \text{ if }i \ne j.
\end{cases} 
\]
To study Questions (1) and (2) for $W$, we will make use of the following notions of thermal capacity and energy, which are natural 
extensions of (1.7) and (3.6) in \cite{KX2015} for $N = 1$.

\begin{definition}\label{thermal:cap}
Let $E \subset \R^N$ and $F \subset \R^d$ be compact sets.
For $\gamma \ge 0$, the $\gamma$-thermal capacity of the set $E\times F$ is defined by
\begin{equation}\label{def:TC}
C_\gamma(E \times F) = \frac{1}{\inf\{ \mathscr{E}_\gamma(\mu) : \mu\in P_0(E\times F)\}},
\end{equation}
where $P_0(E\times F)$ denotes the set of all Borel probability measures $\mu$ on $E\times F$ such that $\mu(\{t\} \times F) = 0$ 
for all $t\in E$ and $\mathscr{E}_\gamma(\mu)$ is the $\gamma$-thermal energy of $\mu$ defined by 
\begin{equation}\label{eq1.1}
\mathscr{E}_\gamma(\mu) = \int_{E \times F} \int_{E\times F} \frac{e^{-\|x - y\|^2/(2\|s - t\|)}}{\|s - t\|^{d/2} \|x - y\|^\gamma} \, 
\mu(ds\, dx) \, \mu(dt\, dy).
\end{equation}
Throughout the paper, $\|\cdot\|$ denotes the Euclidean norm on $\mathbb{R}^N$ or $\mathbb{R}^d$.
\end{definition}

Let $\rho$ denote the parabolic metric on $\mathbb{R}^N \times \mathbb{R}^d$ defined by
\begin{equation}\label{Def:rho} 
\rho((s, x), (t, y)) = \max\{\|s - t\|^{1/2}, \|x - y\|\}.
 \end{equation}
The thermal capacity $C_\gamma(E \times F)$ is related to the parabolic Hausdorff dimension $\dimh (E \times F; \rho)$ of 
$E \times F$ as well as the Hausdorff or packing dimensions of $E$ and $F$ (see Section 2 below).
We denote Hausdorff and packing dimensions by  $\dimh$ and $\dim_{_{\rm P}}$, respectively, and refer to the excellent 
books by Falconer \cite{Fal} and Mattila \cite{Mattila} for their definitions and properties. 

When $F$ has positive Lebesgue measure, it is obvious that $\P(W(E) \cap F \ne \varnothing) > 0$ for every nonempty Borel 
set $E\subset (0,\infty)^N$. When $F$ has Lebesgue measure zero, we obtain a necessary and sufficient condition for 
$\P(W(E) \cap F \ne \varnothing) > 0$ in the following theorem, which gives an answer to Question (1) for the Brownian sheet.

\begin{theorem}\label{T:hit}
Fix $0<a<b<\infty$.
Then, there exists $K=K(N,d,a,b)\in (1,\infty)$ such that for all compact sets $E \subset (a, b)^N$ and $F \subset \{x \in \R^d : \|x\| < b\}$, 
where $F$ has Lebesgue measure zero, we have
\begin{align}\label{hp:bd}
\frac{1}{K} C_0(E\times F) \le \P(W(E) \cap F \ne \varnothing) \le K C_0(E\times F).
\end{align}
As a consequence, $\P(W(E) \cap F \ne \varnothing) > 0$ if and only if $C_0(E \times F) > 0$, and we have
\begin{equation}\label{Eq:TW}
\P(W(E) \cap F \ne \varnothing) \begin{cases}
> 0 & \text{ if } \dimh (E \times F; \rho) > d,\\
= 0 & \text{ if } \dimh (E \times F; \rho) < d.
\end{cases}
 \end{equation}
\end{theorem}

When $W$ is $d$-dimensional Brownian motion, \eqref{Eq:TW} was proved by Taylor and Watson \cite{Taylor:TW}. 
Unlike the case of Brownian motion, which has 
stationary and independent increments, we need to consider the pinned Brownian sheet introduced in \cite{KX2007} and make use of its properties to prove this result.
The commuting property of the filtration associated with the Brownian sheet is also needed in order to apply Cairoli's maximal inequality 
in the proof.

The Hausdorff dimension of $W(E) \cap F$ is a random variable that is in general nonconstant on the event $\{W(E) \cap F \ne 
\varnothing\}$, as was pointed out in \cite{KX2015}.
To answer Question (2), we compute the $L^\infty(\P)$-norm of $\dimh(W(E)\cap F)$.
When $F$ has positive Lebesgue measure, a simple formula is available (see Proposition \ref{prop4.5} below):
\[\|\dimh(W(E)\cap F)\|_\infty = \min\{d, 2\dimh{E}\}. \]
In general, we have the following result:

\begin{theorem}\label{T:dim:cap}
For any compact sets $E \subset (0, \infty)^N$ and $F \subset \mathbb{R}^d$, regardless of whether $F$ has zero or positive Lebesgue measure, we have
\begin{equation}\label{eq1.4}
\|\dimh (W(E) \cap F)\|_\infty = \sup\{ \gamma > 0 : C_\gamma(E \times F) > 0 \}.
\end{equation}
\end{theorem}

This result extends Theorem 1.3 of \cite{KX2015}.
The proof of \eqref{eq1.4} is based on a codimension argument and a hitting probability result for additive stable processes.
The codimension argument allows us to find the Hausdorff dimension of $W(E) \cap F$ by checking whether this set intersects the range of an $n$-parameter $d$-dimensional additive $\alpha$-stable process $X$ independent of $W$.
This reduces the proof of Theorem \ref{T:dim:cap} to determining when $\P(W(E) \cap F \cap X(\mathbb{R}^n \setminus\{0\}) \ne \varnothing)$ is positive (see Proposition \ref{thm6.1} below).

The formula \eqref{eq1.4} holds 
in all dimensions $N \ge 1$ and $d \ge 1$. 
Our next result provides an alternative formula when $d \ge 2N$, and extends Theorem 1.1 of \cite{KX2015}.

\begin{theorem}\label{T:dim:dim-d}
If $d \ge 2N$, then
\begin{equation}\label{eq1.5}
 \|\dimh (W(E) \cap F)\|_\infty = \dimh (E \times F; \rho) - d
\end{equation}
for any compact sets $E \subset (0,\infty)^N$ and $F\subset \R^d$.
\end{theorem}

Theorems \ref{T:dim:cap} and \ref{T:dim:dim-d} together imply the following ``parabolic Frostman theorem''.

\begin{corollary}
If $d \ge 2N$, then
\begin{align}\label{pa:Frostman}
\sup\{\gamma > 0 : C_\gamma(E\times F) > 0\}
= \dimh (E\times F; \rho) - d
\end{align}
for any compact sets $E \subset (0,\infty)^N$ and $F\subset \R^d$.
\end{corollary}

In the case that $d < 2N$, \eqref{eq1.5} does not hold in general, as was pointed out in \cite{KX2015} for $N=1$. In this case,
Proposition \ref{thm2.3} and Remark \ref{R:gamma*} below may be useful for evaluating or estimating the quantity 
$\sup\{ \gamma > 0 : C_\gamma(E \times F) > 0 \}$ in \eqref{eq1.4}. It is also possible to derive a partial result by 
extending the uniform dimension result of Balka and Peres \cite{BP17} to the Brownian sheet and adapting the argument 
in Erraoui and Hakiki \cite{EH25}. Namely, we can prove that if $d < 2N$ and the Assouad dimension of $E$ is at most $d/2$, then 
\eqref{eq1.5}  still holds. Since the proof of this result will require some space, and the essence of the method differs from that of 
the present paper, we will provide the details elsewhere. Despite these partial results, finding an explicit formula like \eqref{pa:Frostman} 
for this quantity when $d<2N$ remains an open problem even for $N=1$ (i.e., $W$ is Brownian motion). 
We believe that this quantity in general cannot be determined by $\dimh(E\times F; \rho)$ and the dimensions of $E$ and $F$ alone.

The rest of the paper is organized as follows.
In Section 2, we establish some properties of the thermal capacity and bounds for $\sup\{\gamma>0 : C_\gamma(E\times F)>0\}$.
In Section 3, we prove Theorem \ref{T:hit}.
In Section 4, we gather necessary ingredients for the codimension argument and study the intersection of $W(E) \cap F$ with the range of an independent additive stable process.
In Section 5, we consider the Hausdorff dimension of $W(E)\cap F$ and prove Theorems \ref{T:dim:cap} and \ref{T:dim:dim-d}.

Throughout this paper, $c, c_1, c_2, \dots$ denote strictly positive and finite constants. We will use notations such as $c_{\ref{5.8}}$ to 
refer to the specific constant appearing in display \eqref{5.8}, etc.

\section{Thermal capacity}

In this section, we establish some properties of the thermal capacity $C_\gamma(E \times F)$ and give bounds for the quantity 
$\sup\{\gamma>0: C_\gamma(E\times F)>0\}$ in terms of the dimensions of $E$ and $F$.

Recall that $\rho$ denotes the parabolic metric (see \eqref{Def:rho}). 
For any $\alpha\ge 0$ and $A \subset \R^N \times \R^d$, the $\alpha$-dimensional parabolic Hausdorff measure of $A$ is defined by
\[
\mathcal{H}^\alpha(A; \rho) = \lim_{\delta \to 0} \inf\left\{ \sum_{n=1}^\infty (\mathrm{diam}_\rho \,U_n)^\alpha : \text{open cover } (U_n)_{n=1}^\infty \text{ of } A \,, \,\sup_n\, (\mathrm{diam}_\rho \,U_n) \le \delta \right\},
\]
where $\mathrm{diam}_\rho \,U$ denotes the diameter of $U$ with respect to the metric $\rho$.
The parabolic Hausdorff dimension of $A$ is defined by
\[
\dimh(A; \rho) = \inf\{ \alpha \ge 0 : \mathcal{H}^\alpha(A; \rho) = 0 \}.
\]

From the definition of thermal capacity (see Definition \ref{thermal:cap}), it follows that if $C_\gamma(E \times F) > 0$, then 
$C_{\gamma'}(E \times F) > 0$ for all $\gamma' < \gamma$, and if $C_\gamma(E \times F) = 0$, then $C_{\gamma''}
(E \times F) = 0$ for all $\gamma'' > \gamma$.

The following proposition provides a convenient connection between the thermal capacity $C_0$ and the parabolic Hausdorff dimension.
When $\dimh (E\times F; \rho) = d$, either $C_0(E \times F) > 0$ or $C_0(E \times F) = 0$ 
could happen for some choices of $E$ and $F$.
\begin{proposition}\label{prop2.1}
Let $E \subset \mathbb{R}^N$ and $F \subset \mathbb{R}^d$ be compact sets. Then
\[ 
C_0(E \times F) \begin{cases}
> 0 & \text{ if }\dimh (E\times F; \rho) > d,\\
= 0 & \text{ if }\dimh (E \times F; \rho) < d.
\end{cases} \]
\end{proposition}

\begin{proof}
Suppose $\dimh (E \times F; \rho) > d' > d$. By Frostman's lemma (cf. \cite[Theorem 8.17]{Mattila}), we can find a Borel probability 
measure $\mu$ on $E \times F$ and $C \in (0,\infty)$ such that 
\begin{align}\label{frostman}
\mu(B_\rho((s, x), r)) \le C r^{d'} \quad \text{for all $s \in E, x \in F$ and $r>0$ small,}
\end{align}
where $B_\rho((s, x),r)$ denotes the open ball centered at $(s, x)$ with radius $r$ with respect to the metric $\rho$. 
Note that \eqref{frostman} implies that $\mu(\{s\} \times F) = 0$ for all $s \in E$.
Indeed, since $d'>d$, for any $\varepsilon>0$, $F$ can be covered by a (finite or infinite) sequence of open balls $B(x_n, r_n)$ in $\R^d$ 
under the Euclidean metric such that $\sum r_n^{d'} \le \varepsilon$.
Then the balls $B_\rho((s,x_n),r_n)$ form a cover for $\{s\}\times F$ and it follows from \eqref{frostman} that $\mu(\{s\} \times F) \le \sum r_n^{d'} \le \varepsilon$.
This shows that $\mu(\{s\} \times F) = 0$.
Moreover, since $E \times F$ is compact, there exists $n_0 \in \mathbb{Z}$ such that $\rho((t,y),(s,x)) < 2^{-n_0}$ for all $(s,x), (t,y) \in E\times F$, and hence
\begin{align}\begin{split}\label{iint:finite}
&\int_{E\times F} \int_{E\times F} \frac{1}{\max\{ \|s - t\|^{1/2}, \|x - y\| \}^{d}} \, \mu(ds\, dx) \, \mu(dt\, dy)\\
& \le \sum_{n=n_0}^\infty \int_{E \times F} \left[ \int_{\substack{(t,y) \in E\times F:\\ \frac{1}{2^{n+1}} \le \rho((t,y),(s,x)) < \frac{1}{2^n}}} 
\frac{\mu(dt\, dy)}{2^{-(n+1)d}} \right] \mu(ds\, dx) \\
& \le  \sum_{n=n_0}^\infty \int_{E\times F} 2^{(n+1)d} \mu(B_\rho((s,x),2^{-n})) \mu(ds\, dx) < \infty,
\end{split}\end{align}
where the last quantity is finite because of \eqref{frostman} and $d'>d$.
Note that there exists a constant $0 < c < \infty$ such that $e^{-x/2} \le c \,x^{-d/2}$ for all $x \ge 0$, so this together with \eqref{iint:finite} implies that
\[ \iint \frac{e^{-\|x - y\|^2/(2\|s - t\|)}}{\|s - t\|^{d/2}} \, \mu(ds\, dx) \, \mu(dt\, dy) \le \iint \frac{c\, \mu(ds\, dx) \, \mu(dt\, dy)}{\max\{ \|s - t\|^{1/2}, \|x - y\| \}^d} 
 < \infty \]
and hence $C_0(E \times F) > 0$. 

Suppose $\dimh (E \times F; \rho) < d$. 
We assume, towards a contradiction, that $C_0(E \times F) > 0$.
Then we can find a probability measure $\mu\in P_0(E\times F)$ such that 
\[ 
\mathscr{E}_0(\mu) = \int_{E\times F} \int_{E\times F} \frac{e^{-\|x - y\|^2/(2\|s - t\|)}}{\|s - t\|^{d/2}} \, \mu(ds\, dx)\, \mu(dt\, dy) < \infty. 
\]
Let 
\begin{align*}
&G_1 = \left\{ (s, x) \in E \times F : \sup_{n\ge 0}\limsup_{r \to 0} \frac{\mu(C_r^n(s, x))}{(2^{-n}r)^{d}} > 0 \right\},\\
& G_2 = \left\{ (s, x) \in E \times F : \limsup_{r \to 0} \frac{\mu(D_r(s, x))}{r^{d}} > 0 \right\},
\end{align*}
where
\[ 
C_r^n(s, x) = \left\{ (t, y) \in E \times F : \frac{1}{2^{n+1}} < \frac{\|s - t\|^{1/2}}{\|x-y\|} \le \frac{1}{2^n} \text{ and } \|x - y\| \le r \right\}
 \]
and
\[ D_r(s, x) = \left\{ (t, y) \in E \times F : \|x - y\| \le \|s - t\|^{1/2} \le r \right\}. \]
Next, we show that $\mu((E\times F) \setminus (G_1 \cup G_2)) = 1$.
In fact, if $(s, x) \in G_1$, then we can find some $n \ge 0$, $a = a(n) > 0$ and a sequence of positive $r_i = r_i(n) \downarrow 0$ 
such that $\mu(C_{r_i}^n(s, x)) \ge a (2^{-n}r_i)^d$. 
Since $\mu(\{s\}\times F) = 0$, we can find $0<q_i<r_i$ such that $\mu(C_{q_i,r_i}^n(s,x)) \ge a (2^{-n}r_i)^d /2$, where
\[
	C_{q,r}^n(s,x) = \left\{ (t,y) \in E \times F: \frac{1}{2^{n+1}} < \frac{\|s-t\|^{1/2}}{\|x-y\|} \le \frac{1}{2^n} \text{ and } q < \|x-y\| \le r \right\}.
\]
By taking a subsequence if necessary, we may assume that $r_{i+1} < q_i$, so that the sets $C_{q_i,r_i}^n(s, x)$ are disjoint. Then 
\begin{align*}
\int_{E\times F} \frac{e^{-\|x - y\|^2/(2\|s - t\|)}}{\|s - t\|^{d/2}} \, \mu(dt\, dy)
&\ge \sum_{i=1}^\infty \int_{C_{q_i,r_i}^n(s, x)} \frac{e^{-\|x - y\|^2/(2\|s - t\|)}}{\|s - t\|^{d/2}} \, \mu(dt\, dy)\\
& \ge \sum_{i=1}^\infty \frac{e^{-2^n}}{(2^{-n}r_i)^{d}}\, \frac{a (2^{-n}r_i)^{d}}{2} = \infty.
\end{align*}
It follows that $\mu(G_1) = 0$ because $\mathscr{E}_0(\mu) < \infty$ implies that the integral on the left-hand side above is finite for $\mu$-a.e. $(s, x)$.

If $(s, x) \in G_2$, then we can find $a > 0$ and a sequence of positive $r_i \downarrow 0$ such that $\mu(D_{r_i}(s, x)) \ge a r_i^d$. 
Since $\mu(\{s\} \times F) = 0$, we can find $0 < q_i < r_i$ such that $\mu(D_{q_i, r_i}(s, x)) \ge a r_i^d /2 $, where
\[ D_{q, r}(s, x) = \left\{ (t, y) \in E \times F : \|x - y\| \le \|s - t\|^{1/2}
\text{ and } q < \|s - t\|^{1/2} \le r \right\}. \] 
By taking a subsequence if needed, we may assume that $r_{i+1} < q_i$, so that the sets $D_{q_i, r_i}(s, x)$ are disjoint. Then 
\begin{align*}
\int \frac{e^{-\|x - y\|^2/(2\|s - t\|)}}{\|s - t\|^{d/2}} \, \mu(dt\, dy)
& \ge \sum_{i=1}^\infty \int_{D_{q_i, r_i}(s, x)} \frac{e^{-\|x - y\|^2/(2\|s - t\|)}}{\|s - t\|^{d/2}} \, \mu(dt\, dy)\\
& \ge \sum_{i=1}^\infty \frac{e^{-1/2}}{r_i^{d}}\, a r_i^{d}/2 = \infty.
\end{align*}
This implies that $\mu(G_2) = 0$.
Therefore, we have $\mu(G) = 1$, where $G = (E \times F)\setminus (G_1 \cup G_2)$. 

Fix an arbitrary $\lambda > 0$. For any $0 < \delta < 1$, let
\[ G_\delta = \left\{ (s, x) \in G : \mu(C_r^n(s, x)) \le \lambda (2^{-n}r)^d, \mu(D_r(s, x)) \le \lambda\, r^d \text{ for all } n \ge 0, r \in (0, \delta] \right\}. \]
Let $\{U_i\}_{i=1}^\infty$ be an open cover of $E \times F$ (and thus of $G_\delta$) with $\mathrm{diam}_\rho \, U_i \le \delta$. 
Write $r_i = \op{diam}_\rho U_i$. If $U_i$ contains 
a point $(s_i, x_i)$ of $G_\delta$, then
\begin{align*}
U_i \cap (E \times F) &\subset \left\{ (t, y) \in E \times F : \|s_i - t\|^{1/2} \le r_i, \|x_i - y\| \le r_i \right\}\\
&= \bigcup_{n=0}^\infty C_{r_i}^n(s_i, x_i)\cup D_{r_i}(s_i, x_i)
\end{align*}
and hence
\begin{align*}
\mu(U_i \cap (E \times F)) &\le \sum_{n=0}^\infty \mu(C_{r_i}^n(s_i, x_i)) + \mu(D_{r_i}(s_i, x_i))\\
& \le \sum_{n=0}^\infty \lambda (2^{-n}r_i)^d + \lambda \, r_i^d = \lambda \, c \, r_i^d
\end{align*}
for some constant $0 < c < \infty$. It follows that
\[ \mu(G_\delta) \le \sum_{i : U_i \cap G_\delta \ne \varnothing} \mu(U_i \cap (E \times F)) \le \lambda \, c \sum_{i=1}^\infty (\op{diam}_\rho U_i)^d. 
\]
Take infimum over all open covers $\{U_i\}_{i=1}^\infty$ of $E \times F$ with $\mathrm{diam}_\rho \, U_i \le \delta$, and then 
let $\delta \downarrow 0$ along a sequence $\delta_n \downarrow 0$.
Since $G \subset \bigcup_{n=1}^\infty G_{\delta_n}$, we have $1 = \mu(G) \le \lambda \, c \, \mathcal{H}^d(E \times F; \rho)$. Since $\lambda > 0$ 
is arbitrary, we have $\mathcal{H}^d(E \times F; \rho) = \infty$ and hence $d \le \dimh (E\times F; \rho)$, which is a contradiction.
\end{proof}

\begin{proposition}\label{prop2.2}
Let $E \subset \mathbb{R}^N$ and $F \subset \mathbb{R}^d$ be compact sets. Then
\[ \dimh  F + 2 \dimh  E \le \dimh (E \times F; \rho) \le \min\{ \dimh  F + 2 \dim_{_{\rm P}} E, \dim_{_{\rm P}} F + 2 \dimh  E \}. 
\]
\end{proposition}

\begin{proof}
The proof is similar to that of Theorem 8.10 in \cite{Mattila} and is thus omitted.
\end{proof}

\begin{proposition}\label{thm2.3}
Let $E \subset \mathbb{R}^N$, $F \subset \mathbb{R}^d$ be compact sets.
Denote 
\[
\gamma^* = \sup\{ \gamma > 0 : C_\gamma(E \times F) > 0 \},
\]
with the convention that $\sup \varnothing = 0$. 
If $\dimh (E \times F; \rho) \le d$, then $\gamma^* = 0$; if $\dimh (E \times F; \rho) > d$, then
\begin{equation}\label{eq2.1}
 \min\{ \dimh  F, \dimh  F + 2\dimh  E - d \} \le \gamma^* \le \min\{ \dimh F, \dimh (E \times F; \rho) - d\}.
\end{equation}
\end{proposition}

\begin{remark}\label{R:gamma*}
As a result of Propositions \ref{prop2.2} and \ref{thm2.3}, if $\dim_{_{\rm H}} E = \dim_{_{\rm P}} E$ or $\dim_{_{\rm H}} F = \dim_{_{\rm P}} F$  
(e.g., this is the case if $E$ or $F$ is an arbitrary self-similar set, cf. Falconer \cite[p. 551]{Fal89}), then
\begin{equation*}
\sup\{ \gamma > 0 : C_\gamma(E\times F) > 0\} = \max\{ 0, \dim_{_{\rm H}} F + \min\{0, 2\dim_{_{\rm H}} E - d\} \}.
\end{equation*}
\end{remark}

\begin{proof}[Proof of Proposition \ref{thm2.3}]
We first assume that $\dimh (E \times F; \rho) > d$ and derive the lower bound in \eqref{eq2.1}.
Consider two cases.

Case 1: $2 \dimh  E - d > 0$.
In this case, we have to show that $\dimh  F \le \gamma^*$. If $\dimh  F = 0$, there is nothing to prove, so we assume that $\dimh  F > 0$. 
Let $0 < \gamma < \dimh  F$. By Frostman's lemma, we can find Borel probability measures $\mu_1$ on $E$ and $\mu_2$ on 
$F$ such that
\[ \int_E \int_E \frac{\mu_1(ds) \, \mu_1(dt)}{\|s - t\|^{d/2}} < \infty, \quad 
\int_F \int_F \frac{\mu_2(dx) \, \mu_2(dy)}{\|x - y\|^{\gamma}} < \infty. 
\]
Then $\mu := \mu_1 \otimes \mu_2$ is a Borel probability measure on $E \times F$ such that $\mu(\{s\} \times F) = 0$ for every 
$s \in E$ and
\begin{align*}
\iint \frac{e^{-\|x - y\|/(2\|s - t\|)}}{\|s - t\|^{d/2} \|x - y\|^\gamma} \, \mu(ds\, dx) \, \mu(dt\, dy) \le \int_E \int_E \frac{\mu_1(ds) \, 
\mu_1(dt)}{\|s - t\|^{d/2}} \int_F \int_F \frac{\mu_2(dx) \, \mu_2(dy)}{\|x - y\|^{\gamma}} < \infty.
\end{align*}
This implies that $\gamma \le \gamma^*$ for all $0 < \gamma < \dimh  F$. Hence $\dimh  F \le \gamma^*$.

Case 2: $2 \dimh E - d \le 0$.
In this case, we have to show that $\alpha := \dimh  F + 2 \dimh  E - d \le \gamma^*$. There is nothing to prove if $\alpha \le 0$, 
so we assume that $\alpha > 0$. 
Let $\varepsilon \in (0, \alpha)$ be such that $\dimh F + 2\dimh  E - \varepsilon - d > 0$ and take $0 < \gamma < \dimh F + 2\dimh  E 
- \varepsilon - d$.
By Frostman's lemma, we can find Borel probability measures $\mu_1$ on $E$ and $\mu_2$ on $F$ such that
\[ \int_E \int_E \frac{\mu_1(ds) \, \mu_1(dt)}{\|s - t\|^{\dimh E - \varepsilon/2}} < \infty, \quad 
\int_F \int_F \frac{\mu_2(dx) \, \mu_2(dy)}{\|x - y\|^{\gamma + d - 2 \dimh E + \varepsilon}} < \infty. \]
Set $\mu = \mu_1 \otimes \mu_2$.
Note that there is a constant $0 < c < \infty$ such that 
\[ \exp\left({-\frac{\|x - y\|^2}{2\|s - t\|}}\right) \le c \left(\frac{\|s - t\|}{\|x - y\|^2}\right)^{\frac{d}{2} - \dimh E + \frac{\varepsilon}{2}} 
\]
for all $s, t \in E$ and $x, y \in F$. It follows that
\begin{align*}
& \int_{E \times F} \int_{E\times F} \frac{e^{-\|x - y\|^2/(2\|s - t\|)}}{\|s - t\|^{d/2} \|x - y\|^\gamma} \, \mu(ds \, dx) \, \mu(dt\, dy)\\
\le & \, c \int_E \int_E \frac{\mu_1(ds) \, \mu_1(dt)}{\|s - t\|^{\dimh  E - \varepsilon/2}} \times \int_F \int_F \frac{\mu_2(dx) \,
 \mu_2(dy)}{\|x - y\|^{\gamma + d - 2\dimh E + \varepsilon}} < \infty.
\end{align*}
So, $\gamma \le \gamma^*$ for all $0 < \gamma < \dimh F + 2\dimh E - \varepsilon - d$, and hence $\dimh F + 2\dimh E - 
\varepsilon - d \le \gamma^*$. Letting $\varepsilon \to 0$ yields $\dimh F + 2\dimh E - d \le \gamma^*$.

Next, we derive the upper bound in in \eqref{eq2.1} by showing $\gamma^* \le \dimh  F$ and $\gamma^* \le \dimh (E \times F; \rho) - d$.
If $\gamma^* = 0$, there is nothing to prove, so we assume $\gamma^* > 0$.
Let $0 < \gamma < \gamma^*$, so that $\mathscr{E}_\gamma(\mu) < \infty$ for some $\mu \in P_0(E\times F)$. 
Let 
\begin{align*}
&G_0 = \left\{ (s, x) \in E \times F : \limsup_{r \to 0} \frac{\mu(B_r(s, x))}{r^\gamma} > 0\right\},\\
&G_1 = \left\{ (s, x) \in E \times F : \sup_{n\ge 0}\limsup_{r \to 0} \frac{\mu(C_r^n(s, x))}{(2^{-n})^d\, r^{d + \gamma}} > 0 \right\},
\end{align*}
where
\[ B_r(s, x) = \left\{ (t, y) \in E \times F : \|x - y\| \le \|s - t\|^{1/2} \text{ and } \|x - y\| \le r \right\} \]
and
\[ C_r^n(s, x) = \left\{ (t, y) \in E \times F : \frac{1}{2^{n+1}}<\frac{\|s - t\|^{1/2}}{\|x-y\|} \le \frac{1}{2^n} \text{ and } \|x - y\| \le r \right\}. \]
Let us show that $\mu((E\times F) \setminus(G_0 \cup G_1)) = 1$.
In fact, if $(s, x) \in G_0$, we can find $a > 0$ and a sequence of positive $r_i \downarrow 0$ such that $\mu(B_{r_i}(s, x)) \ge a r_i^\gamma$. 
Note that $\mu(E \times \{x\}) = 0$ since $\mathscr{E}_\gamma(\mu) < \infty$. So we can find $0 < q_i < r_i$ such that 
$\mu(B_{q_i, r_i}(s, x)) \ge a r_i^\gamma /2$, where
\[ B_{q, r}(s, x) =  \left\{ (t, y) \in E \times F : \|x - y\| \le \|s - t\|^{1/2} \text{ and } q < \|x - y\| \le r \right\}. \]
By taking a subsequence if needed, we may assume that $r_{i+1} < q_i$, so that the sets $B_{q_i, r_i}(s, x)$ are disjoint. Moreover, since 
$E$ is compact, there is $0 < M < \infty$ such that $\|t\| \le M$ for all $t \in E$. Then 
\begin{align*}
\int_{E\times F} \frac{e^{-\|x - y\|^2/(2\|s - t\|)}}{\|s - t\|^{d/2} \|x - y\|^\gamma} \, \mu(dt\, dy)
& \ge \sum_{i=1}^\infty \int_{B_{q_i, r_i}(s, x)} \frac{e^{-\|x - y\|^2/(2\|s - t\|)}}{\|s - t\|^{d/2} \|x - y\|^\gamma} \, \mu(dt\, dy)\\
& \ge \sum_{i=1}^\infty \frac{e^{-1/2}}{(2M)^{d/2}\, r_i^{\gamma}}\, a r_i^\gamma /2 = \infty.
\end{align*}
This shows that $\mu(G_0) = 0$ because $\mathscr{E}_\gamma(\mu) < \infty$ implies that the above integral is finite for $\mu$-a.e.~$(s, x)$. 
Also, we can show that $\mu(G_1) = 0$ as in the proof of Proposition \ref{prop2.1}.
Hence, we have $\mu(G) = 1$, where $G = (E \times F)\setminus (G_0 \cup G_1)$. 

Fix an arbitrary $\lambda > 0$.
For each $0 < \delta <1$, let 
\[ G_\delta = \left\{ (s, x) \in G : \mu(B_r(s, x)) \le \lambda \, r^\gamma, \mu(C_r^n(s, x)) \le \lambda (2^{-n})^d\, r^{d + \gamma} 
\text{ for all } n \ge 0, r \in 
(0, \delta] \right\}. 
\]
Let $\{U_i\}_{i=1}^\infty$ be an open cover of $F$ with $r_i := \op{diam} U_i \le \delta$. 
Then $\{E \times (U_i \cap F)\}_{i=1}^\infty$ is a cover of 
$E \times F$, and thus of $G_\delta$. If $E \times (U_i \cap F)$ contains a point $(s_i, x_i)$ of $G_\delta$, then 
\begin{align*}
 E \times (U_i \cap F) & \subset \left\{ (t, y) \in E \times F : \|x_i - y\| \le r_i \right\}
= B_{r_i}(s_i, x_i) \cup \bigcup_{n=0}^\infty C_{r_i}^n (s_i, x_i).
\end{align*}
Since $\mu(\{s_i\} \times F) = 0$ and $(s_i, x_i) \in G_\delta$, we have
\begin{align*}
\mu(E \times (U_i \cap F)) & \le \mu(B_{r_i}(s_i, x_i)) + \sum_{n=0}^\infty \mu(C_{r_i}^n(s_i, x_i))\\
& \le \lambda \, r_i^\gamma + \sum_{n=0}^\infty \lambda (2^{-n})^d \, r_i^{d + \gamma} \le \lambda \, c \, r_i^\gamma,
\end{align*}
for some constant $0 < c < 1$.
It follows that
\[ \mu(G_\delta) \le \sum_{i : (E \times U_i) \cap G_\delta \ne \varnothing} \mu(E \times (U_i \cap F)) \le \lambda\, c \sum_{i=1}^\infty (\op{diam}U_i)^\gamma. \]
Taking infimum over all open covers $\{U_i\}_{i=1}^\infty$ of $F$ with $\mathrm{diam} U_i \le \delta$, and letting $\delta \downarrow 0$, we have $1=  \mu(G) \le \lambda \, c\,  \mathcal{H}^\gamma(F)$. Since $\lambda > 0$ is arbitrary, we conclude that $\mathcal{H}^\gamma(F) = \infty$ and $\gamma \le \dimh F$.  Hence $\gamma^* \le \dimh F$.

To show $\gamma^* \le \dimh (E \times F; \rho) - d$, we may assume $\gamma^* > 0$. Let $0 < \gamma < \gamma^*$ so that 
$\mathscr{E}_\gamma(\mu) < \infty$ for some Borel probability measure $\mu \in P_0(E\times F)$.
Let $G_1$ be as in the last paragraph and let
\[ G_2 = \left\{ (s, x) \in E \times F : \varlimsup_{r \to 0} \frac{\mu(D_r(s, x))}{r^{d + \gamma}} > 0 \right\}, \]
where
\[ D_r(s, x) = \left\{ (t, y) \in E \times F : \|x - y\| \le \|s - t\|^{1/2} \le r \right\}. \]
Then we can argue as in the second part of Proposition \ref{prop2.1} to show that $\mathcal{H}^{d + \gamma}(E \times F; \rho) = \infty$ and thus $d + \gamma \le \dimh (E \times F; \rho)$. Hence $\gamma^* \le \dimh (E \times F; \rho) - d$.

Finally, let us assume $\dimh (E \times F; \rho) \le d$ and prove that $\gamma^* = 0$. 
Suppose, towards a contradiction, that $\gamma^* > 0$. Then we can take $0 < \gamma < \gamma^*$ so that $\mathscr{E}_\gamma(\mu)<\infty$ for some Borel probability measure $\mu \in P_0(E\times F)$. Proceeding as in the last paragraph, we can show 
that $d + \gamma \le \dimh (E \times F; \rho)$, that is, $\gamma \le \dimh (E \times F; \rho) - d \le 0$, which is a contradiction. Hence $\gamma^* = 0$.
\end{proof}

\section{Proof of Theorem \ref{T:hit}}

The goal of this section is to prove Theorem \ref{T:hit}. 
In order to prove this result, we recall the pinned Brownian sheet and its properties. 
Readers may refer to \cite{KX2007} for details. We will also prove a technical lemma (Lemma \ref{lemma4.2}) 
which is an extension of Proposition 4.2 in \cite{KX2015}.

For any fixed $s \in \mathbb{R}^N$, the associated pinned Brownian sheet is defined by
\begin{align}\label{pinned:BS}
W_s(t) = W(t) - C_{s, t} W(s), \quad t \in \mathbb{R}^N_+,
\end{align}
where
\[ C_{s, t} = \prod_{i=1}^N\left(\frac{s_i \wedge t_i}{s_i}\right). \]
Note that $0 \le C_{s, t} \le 1$, and $W_s(t)$ and $W(s)$ are independent because they are Gaussian and $\E[W_s(t)W(s)] = 0$. 
It is also helpful to notice that $C_{s, t} W(s)$ is $\E(W(t)|W(s))$. Hence $W_s(t) = W(t) - \E(W(t)|W(s))$ which is independent 
of $W(s)$. 

For each $\pi \subset \{1, \dots, N\}$, define the partial ordering $\preccurlyeq_\pi$ on $\mathbb{R}^N$ by
\[ s \preccurlyeq_\pi t \Longleftrightarrow \begin{cases}
s_i \le t_i & \text{ if } i \in \pi\\
s_i \ge t_i & \text{ if } i \notin \pi.
\end{cases} \]
Define $\mathscr{F}_\pi(s) = \sigma\{W(u) : u \preccurlyeq_\pi s\}$. Then $\{W_s(t) : s \preccurlyeq_\pi t \}$ is independent 
of $\mathscr{F}_\pi(s)$ (see \cite[Lemma 5.1]{KX2007}).
The filtration $\{ \mathscr{F}_\pi(s) : s \in \mathbb{R}^N_+ \}$ is commuting in the sense of Cairoli, i.e., for any $s, t \in \mathbb{R}^N_+$,
 $\mathscr{F}_\pi(s)$ and $\mathscr{F}_\pi(t)$ are conditionally independent given $\mathscr{F}_\pi(s \wedge_\pi t)$, where
\[ (s \wedge_\pi t)_i := \begin{cases}
s_i \wedge t_i & \text{ if } i \in \pi\\
s_i \vee t_i & \text{ if } i \notin \pi.
\end{cases} \]
Equivalently, $\{ \mathscr{F}_\pi(s) : s \in \mathbb{R}^N_+ \}$ is commuting if and only if for any $s, t \in \mathbb{R}^N_+$ and any 
bounded $\mathscr{F}_\pi(t)$-measurable function $Y$, $\E(Y|\mathscr{F}_\pi(s)) = \E(Y|\mathscr{F}_\pi(s \wedge_\pi t))$. 
The commuting property allows us to apply Cairoli's maximal inequality for martingales \cite[Corollary 3.2]{KX2007}: For any $Z \in L^2(\P)$,
\begin{align}\label{E:Cairoli}
\E\left(\sup_{t \in \mathbb{Q}^N_+} [\E(Z|\mathscr{F}_\pi(t))]^2\right) \le 4^N \E(Z^2). 
\end{align}
We refer to \cite{K} for more details about commuting filtrations and Cairoli's maximal inequalities for multiparameter processes.

For a finite Borel measure $\nu$ on $\mathbb{R}^d$, its Fourier transform is defined by
\[ \widehat{\nu}(\xi) = \int_{\mathbb{R}^d} e^{i\langle \xi, x\rangle} \, \nu(dx) , \quad \xi \in \mathbb{R}^d.\]
Note that we use the convention $e^{i\langle \xi, x \rangle}$ for the Fourier transform.
This convention also applies to the Fourier transform of other functions or tempered distributions in the rest of the paper.
Recall that a function $f: \R^d \to [0,+\infty]$ is said to be tempered if it is measurable and
\[
	\int_{\R^d} \frac{f(x)}{(1+\|x\|)^n}dx < \infty \quad \text{for some $n \ge 0$}.
\]
A function $f: \R^d \to [0,+\infty]$ is said to be positive definite if it is tempered and for all rapidly decreasing test functions $\phi: \R^d \to \R$,
\[
	\int_{\R^d} \int_{\R^d} f(x-y) \phi(x) \phi(y) dy dx \ge 0.
\]
Let us recall the following lemma.

\begin{lemma}\cite[Lemma 4.1]{KX2015}
\label{lemma4.1}
If $f: \mathbb{R}^d \to [0, +\infty]$ is lower semicontinuous and positive definite, 
then for all finite Borel measures $\rho$ on $\mathbb{R}^d$,
\begin{align*}
\int_{\R^d} \int_{\R^d} f(x - y) \, \rho(dx) \, \rho(dy) = \frac{1}{(2\pi)^d} \int_{\mathbb{R}^d} \widehat{f}(\xi)\, |\widehat{\rho}(\xi)|^2 \, d\xi.
\end{align*}
If $f$ is in addition bounded, then for all finite Borel measures $\rho$ and $\nu$ on $\mathbb{R}^d$,
\begin{align*}
\int_{\R^d} \int_{\R^d} f(x - y) \, \rho(dx) \, \nu(dy) = \frac{1}{(2\pi)^d} \int_{\mathbb{R}^d} \widehat{f}(\xi) \, \widehat{\rho}(\xi) \, \overline{\widehat{\nu}(\xi)} 
\, d\xi.
\end{align*}
\end{lemma}

\begin{lemma}\label{lem:mu_n}
Let $K$ be a compact set in $\R^m$ and $\mu$ be a Borel probability measure on $K$.
Then there exists a sequence of probability measures $(\mu_n)_{n=1}^\infty$ on $K$, where each $\mu_n$ is a finite convex combination of delta measures $\delta_x$ where $x \in K$, such that $\mu_n$ converges weakly to $\mu$.
\end{lemma}

\begin{proof}
By the Riesz representation theorem \cite[III. 5.7]{Conway}, the space $M(K)$ of regular Borel signed measures on $K$ with the total 
variation norm can be identified as the dual space $C_0(K)^*$ of $C_0(K)$, where $C_0(K)$ is the Banach space of all bounded continuous 
functions $f: K \to \R$ such that for every $\varepsilon>0$, $\{x \in K : |f(x)| \ge \varepsilon\}$ is compact,
equipped with the $\sup$ norm.
Also, since $K$ is a metric space, every Borel probability measure on $K$ is regular (see \cite[Theorem 1.1]{Billingsley}). 
Hence, the set $P(K)$ of all Borel probability measures on $K$ satisfies $P(K) \subset M(K) \cong C_0(K)^*$.
By the Banach--Alaoglu theorem \cite[V. 3.1]{Conway}, the closed unit ball $B$ in $M(K) \cong C_0(K)^*$ is compact in the weak star topology of $C_0(K)^*$.
Note that $P(K)$ is closed in the weak star topology (this is because the compactness of $K$ implies ${\bf 1}_K \in C_0(K)$, so if $\mu_n \in P(K)$ converges to $\mu$ in the weak star topology, then $\mu(K) = \int {\bf 1}_K d\mu = \lim_{n\to\infty} \int {\bf 1}_K d\mu_n = \lim_{n\to\infty} \mu_n(K) = 1$).
Now $P(K)$, as a weak star closed subset of the weak star compact set $B$, is also weak star compact.
It follows that $P(K)$ is a convex, weak star compact set in $M(K) \cong C_0(K)^*$. By the Krein--Milman theorem \cite[V. 7.4]{Conway}, $P(K)$ is equal to the weak star closure of the convex hull of $E$, where $E$ is the set of extreme points of $P(K)$.
Owing to the compactness of $K$, the weak star topology in $M(K) \cong C_0(K)^*$ is metrizable \cite[V. 5.1]{Conway}.
Therefore, there exists a sequence $\mu_n$ in $\mathrm{conv}(E)$ that converges to $\mu$ in the weak star topology of $C_0(K)^*$, which means precisely that $\mu_n$ converges weakly to $\mu$ in the context of probability measures.
Finally, since $K$ is compact, the set of extreme points $E$ of $P(K)$ coincides with the set of delta measures $\{\delta_x: x \in K\}$ (see \cite[V. 8.4]{Conway}), so each $\mu_n$ is a finite convex combination of delta measures.
This completes the proof of the lemma.
\end{proof}

\begin{lemma}\label{lemma4.2}
Let $0 < c_1 < 1 < c_2$ be constants. For distinct $s, t \in \mathbb{R}^N$ and for any $w \in \R^d$, define
\begin{align*}
p_{s, t}(w) = \frac{e^{-\|w\|^2/(2\|t - s\|)}}{\|t - s\|^{d/2}}, \quad 
q_{s, t}(w) = \frac{e^{-\|w\|^2/(2c_1\|t - s\|)}}{\|t - s\|^{d/2}}, \quad
& r_{s, t}(w) = \frac{e^{-\|w\|^2/(2c_2\|t - s\|)}}{ \|t - s\|^{d/2}}.
\end{align*}
Let $E \subset \mathbb{R}^N$, $F \subset \mathbb{R}^d$ be compact sets and $\mu \in P_0(E\times F)$.
Let $\phi_1$, $\phi_2 \in C_0(\mathbb{R}^d) \cap L^1(\mathbb{R}^d)$ be positive definite functions with $0 \leq \widehat{\phi}_1
 \leq \widehat{\phi}_2$. Then
\begin{align}
\begin{aligned}
\label{4.2}
\iint (\phi_1 \ast p_{s, t})(x - y) \, \mu(ds\, dx) \, \mu(dt\, dy) 
& \leq \iint (\phi_2 \ast p_{s, t})(x - y)\, \mu(ds\, dx) \, \mu(dt\, dy)
\end{aligned}
\end{align}
and
\begin{align}
\begin{aligned}\label{4.3}
&c_2^{-(1+\frac{d}{2})} \iint (\phi_1 \ast r_{s, t})(x - y)\, \mu(ds\, dx)\mu(dt\, dy)  \\
&\hspace{30pt} \leq \iint (\phi_1 \ast p_{s, t})(x - y) \, \mu(ds\, dx) \, \mu(dt\, dy) \\
& \hspace{60pt} \leq c_1^{-(1+\frac{d}{2})} \iint (\phi_1 \ast q_{s, t})(x - y) \, \mu(ds\, dx) \, \mu(dt\, dy).
\end{aligned}
\end{align}
\end{lemma}

\begin{proof}
Thanks to Lemma \ref{lem:mu_n}, we can find a sequence $(\mu_n)_{n=1}^\infty$ of probability measures on $E \times F$ that converges weakly to $\mu$, where each $\mu_n$ is of the form $\sum_{i=1}^m a_i \, \delta_{s_i} \otimes \delta_{x_i}$ where $s_i \in E$, $x_i \in F$, $a_i > 0$, $\sum_{i=1}^m a_i = 1$, and $m$, $a_i$, $s_i$, $x_i$ may depend on $n$.

Let $\nu = \sum_{i=1}^m a_i \, \delta_{s_i} \otimes \delta_{x_i}$. 
Define $p_{s,t} = (2\pi)^{d/2} \, \delta_0$ when $s = t$.
By the assumptions, for fixed $s, t \in E$, the functions $x \mapsto (\phi_1 \ast p_{s, t})(x)$ and $x \mapsto (\phi_2 \ast p_{s, t})(x)$ are positive definite, bounded and continuous. By Fubini's theorem and Lemma \ref{lemma4.1}, for $k=1,2$,
\begin{align*}
& \quad \iint (\phi_k \ast p_{s, t})(x - y) \, \nu(ds\, dx) \, \nu(dt\, dy)\\
& = \sum_{i, j = 1}^m a_i a_j \iint \left[ \iint (\phi_k \ast p_{s, t})(x - y) \, \delta_{x_i}(dx) \, \delta_{x_j}(dy)\right] \delta_{s_i}(ds) \,
\delta_{s_j}(dt)\\
& = \frac{1}{(2\pi)^d} \sum_{i, j = 1}^m a_i a_j \iint \left[ \int_{\mathbb{R}^d} \widehat{\phi}_k(\xi)\, \widehat{p}_{s, t}(\xi) \, 
\widehat{\delta}_{x_i}(\xi) \, \overline{\widehat{\delta}_{x_j}(\xi)} \, d\xi \right] \delta_{s_i}(ds) \,\delta_{s_j}(dt).
\end{align*}
Note that $\widehat{p}_{s, t}(\xi) = (2\pi)^{d/2} \, P_\xi(s - t)$, where $P_\xi(s) = e^{-\frac{1}{2}\|\xi\|^2\|s\|}$. This holds even 
when $s = t$. Then by Fubini's theorem, 
\begin{align*}
& \quad \iint (\phi_k \ast p_{s, t})(x - y) \, \nu(ds\, dx) \, \nu(dt\, dy)\\
& =\frac{1}{(2\pi)^{d/2}} \sum_{i, j = 1}^m a_i a_j \int_{\mathbb{R}^d} \left[ \iint P_\xi(s - t) \, \delta_{s_i}(ds) \,\delta_{s_j}(dt) \right] \widehat{\phi}_k(\xi)\, \widehat{\delta}_{x_i}(\xi) \, \overline{\widehat{\delta}_{x_j}(\xi)} \,  d\xi.
\end{align*}
Observe that $s \mapsto P_\xi(s)$ is the characteristic function of a multivariate Cauchy distribution, which is positive definite.
Indeed, for $\xi \in \mathbb{R}^d \setminus\{0\}$,
\begin{align}\label{3.7}
\widehat{P}_\xi(\eta) = \frac{(2\pi)^N K \cdot \frac{1}{2}\|\xi\|^2}{(\frac{1}{4}\|\xi\|^4 + \|\eta\|^2)^{\frac{N+1}{2}}}, \quad \eta \in 
\mathbb{R}^N,
\end{align}
where $K = \pi^{-\frac{N+1}{2}}\,\Gamma(\frac{N+1}{2})$ (see e.g. \cite[Example 2.5.3]{ST}).
Then by Lemma \ref{lemma4.1} again,
\begin{align*}
& \quad \iint (\phi_k \ast p_{s, t})(x - y) \, \nu(ds\, dx) \, \nu(dt\, dy)\\
& =  \frac{1}{(2\pi)^{N+\frac{d}{2}}} \sum_{i, j = 1}^m a_i a_j \int_{\mathbb{R}^d} \left[ \int_{\mathbb{R}^N}\widehat{P}_\xi(\eta) \, \widehat{\delta}_{s_i}(\eta) \,\overline{\widehat{\delta}_{s_j}(\eta)} \, d\eta \right] \widehat{\phi}_k(\xi)\, \widehat{\delta}_{x_i}(\xi) \, \overline{\widehat{\delta}_{x_j}(\xi)} \,  d\xi\\
& = \frac{1}{(2\pi)^{N+\frac{d}{2}}} \iint \widehat{\phi}_k(\xi)\, \widehat{P}_\xi(\eta) \left| \sum_{i=1}^m a_i \, \widehat{\delta}_{s_i}(\eta) \widehat{\delta}_{x_i}(\xi) \right|^2 d\xi \, d\eta.
\end{align*}
This together with $\widehat{P}_\xi(\eta) \geq 0$ and $0 \leq \widehat{\phi}_1(\xi) \leq \widehat{\phi}_2(\xi)$ implies that
\begin{align}\label{6.6}
\iint (\phi_1 \ast p_{s, t})(x - y) \, \mu_n(ds\, dx) \, \mu_n(dt\, dy) \leq \iint (\phi_2 \ast p_{s, t})(x - y)\, \mu_n(ds\, dx) \, \mu_n(dt\, dy)
\end{align}
for all $n$. 
For $k=1, 2$, since $\phi_k$ is bounded and uniformly continuous, the function
\begin{align*}
(s, x, t, y) \in (E \times F)^2 \mapsto (\phi_k \ast p_{s, t})(x - y)
\end{align*}
is continuous. Hence, by the convergence of $\mu_n$, we can let $n \to \infty$ in \eqref{6.6} to obtain \eqref{4.2}.

To prove \eqref{4.3}, define $q_{s, t} = (2\pi c_1)^{d/2} \, \delta_0$, $r_{s, t} = (2\pi c_2)^{d/2} \, \delta_0$ when $s = t$, and notice that
$\widehat{q}_{s, t}(\xi) = (2\pi)^{d/2} \, Q_\xi(s - t)$ and $\widehat{r}_{s, t}(\xi) = (2\pi)^{d/2} \, R_\xi(s - t)$, where
\begin{align*}
Q_\xi(s) = c_1^{d/2} \, e^{-\frac{1}{2}c_1\|\xi\|^2 \|s\|}, \quad
R_\xi(s) = c_2^{d/2} \, e^{-\frac{1}{2}c_2\|\xi\|^2 \|s\|}
\end{align*}
for $\xi \in \mathbb{R}^d,  s \in \mathbb{R}^N$.
As in \eqref{3.7}, for $\xi \in \mathbb{R}^d \setminus\{0\}$,
\begin{align*}
& \widehat{Q}_\xi(\eta) =  c_1^{1+\frac{d}{2}} \, \frac{(2\pi)^N K \cdot \frac{1}{2}\|\xi\|^2}{(\frac{1}{4}c_1^2\|\xi\|^4 + \|\eta\|^2)^{\frac{N+1}{2}}}, \quad
\widehat{R}_\xi(\eta) = c_2^{1+ \frac{d}{2}} \, \frac{(2\pi)^N K \cdot \frac{1}{2}\|\xi\|^2}{(\frac{1}{4}c_2^2\|\xi\|^4 + \|\eta\|^2)^{\frac{N+1}{2}}}
\end{align*}
for $\eta \in \mathbb{R}^N$.
Setting $\nu = \sum_{i=1}^m a_i \, \delta_{s_i} \otimes \delta_{x_i}$, we can compute as before to get
\begin{align*}
 \iint (\phi_1 \ast q_{s, t})(x - y) \nu(ds\, dx) \nu(dt\, dy) = \frac{1}{(2\pi)^{N+\frac{d}{2}}} \iint  \widehat{\phi}_1(\xi)  \widehat{Q}_\xi(\eta) \left| \sum_{i=1}^m a_i  \widehat{\delta}_{s_i}(\eta) \widehat{\delta}_{x_i}(\xi) \right|^2 d\xi \, d\eta
\end{align*}
and
\begin{align*}
 \iint (\phi_1 \ast r_{s, t})(x - y)  \nu(ds\, dx) \nu(dt\, dy) = \frac{1}{(2\pi)^{N+\frac{d}{2}}} \iint  \widehat{\phi}_1(\xi) \widehat{R}_\xi(\eta) \left| \sum_{i=1}^m a_i  \widehat{\delta}_{s_i}(\eta) \widehat{\delta}_{x_i}(\xi) \right|^2  d\xi \, d\eta.
\end{align*}
Since $0 < c_1 < 1 < c_2$, we have 
\[0 \leq c_2^{-(1+ \frac{d}{2})}\, \widehat{R}_\xi(\eta) \leq \widehat{P}_\xi(\eta) \leq c_1^{-(1+ \frac{d}{2})}\, \widehat{Q}_\xi(\eta).\]
Again we can use the arguments in the first part to deduce \eqref{4.3}.
\end{proof}

\begin{proof}[Proof of Theorem \ref{T:hit}]
Throughout the proof, for distinct $s, t \in E$, we denote
\begin{align}\label{E:pst}
p_{s, t}(w) = \frac{e^{-\|w\|^2/(2\|s - t\|)}}{\|t - s\|^{d/2}}, \quad w \in \mathbb{R}^d.
\end{align}

\medskip

\noindent
\textit{Part 1: Lower bound.}
Let $\mu \in P_0(E\times F)$.
Let $\varepsilon \in (0, 1)$, $f_{\varepsilon} = |B(0, \varepsilon)|^{-1} {\bf 1}_{B(0, \varepsilon)}$, where $|\cdots|$ denotes Lebesgue measure, and let $\phi_\varepsilon = f_{\varepsilon} \ast f_{\varepsilon}$.
Define
\begin{align}
 Z_{\varepsilon}(\mu) = \int_{E\times F} \phi_\varepsilon(W(s) - x) \,\mu(ds\,dx).
\end{align}
Let $F^\varepsilon$ denote the $\varepsilon$-neighborhood of $F$, that is,
\begin{align}
F^\varepsilon = \left\{ x \in \R^d : \inf_{y\in F} \|x-y\| < \varepsilon \right\}.
\end{align}
By the Cauchy--Schwarz inequality,
\begin{align}\label{3.16}
\P(W(E) \cap F^\varepsilon \neq \varnothing) \geq 
\P(Z_{\varepsilon}(\mu) > 0) \geq \frac{\E[Z_\varepsilon(\mu)]^2}{\E[Z_\varepsilon(\mu)^2]}.
\end{align}
Note that
\begin{align}\label{E:phi_eps:bd}
2^{-d} f_{\varepsilon/2} \leq \phi_\varepsilon \leq 2^d f_{2\varepsilon}.
\end{align}
By \eqref{E:phi_eps:bd} and the definition of $W_s(t)$ in \eqref{pinned:BS},
\begin{align}
\begin{aligned}
\E[Z_\varepsilon(\mu)^2] & = \iint \E[\phi_\varepsilon(W(s) - x) \phi_\varepsilon(W(t) - y)] \,\mu(ds\,dx)\,\mu(dt\,dy)\\
& \leq 4^d \iint \E[f_{2\varepsilon}(W(s) - x) f_{2\varepsilon}(W_s(t) + C_{s, t}W(s) - y)] \,\mu(ds\,dx) \, \mu(dt\,dy).
\end{aligned}
\end{align}
By \cite[(6.17) and (5.10)]{KX2007}, there exist constants $c_{\ref{5.8}}$, $c_{\ref{5.9}}' > 1$ and $0 < c_{\ref{5.9}} < 1$, depending 
on $a, b$ and $N$ only, such that
\begin{align}
\label{5.8}
1 - C_{s, t} & \leq c_{\ref{5.8}}\, \|t - s\|^{1/2};\\
\label{5.9}
c_{\ref{5.9}}\, \|t - s\| & \leq \sigma_{s, t}^2 \leq c_{\ref{5.9}}' \, \|t - s\|.
\end{align}
Note that $\|W(s) - x\| \leq 2\varepsilon$ and $\|W_s(t) + C_{s, t}W(s) - y\| \leq 2 \varepsilon$ imply $\|W_s(t) + C_{s, t}x - y\| \leq 
\|W_s(t) + C_{s, t} W(s) - y\| + \|C_{s, t}(W(s) - x)\| \leq 4\varepsilon$.
This together with independence of $W(s)$ and $W_s(t)$ and \eqref{E:phi_eps:bd} imply that
\begin{align}
\begin{aligned}\label{7.9}
& \E[f_{2\varepsilon}(W(s) - x) f_{2\varepsilon}(W_s(t) + C_{s, t}W(s) - y)]\\
& \quad \leq c_{\ref{7.9}} \E[f_{2\varepsilon}(W(s) - x)] \E[ \phi_{8\varepsilon}(W_s(t) + C_{s, t}x - y)].
\end{aligned}
\end{align}
Moreover, we have
\begin{align}
\begin{aligned}
\E[f_{2\varepsilon}(W(s) - x)] \le (2\pi a^N)^{-d/2}.
\end{aligned}
\end{align}
It follows that 
\begin{align}\label{7.12}
&\E[Z_\varepsilon(\mu)^2]\\
\notag
& \leq  c_{\ref{7.12}} \iint \left[\int_{\mathbb{R}^d} \phi_{8\varepsilon}(w + C_{s, t}x - y) \frac{1}{\|t - s\|^{d/2}} \exp\left(-\frac{\|w\|^2}{2 \,c_{\ref{5.9}}'\|t - s\|}\right) dw \right]\mu(ds\,dx)\, \mu(dt\,dy)\\
\notag
& = c_{\ref{7.12}} \iint \left[\int_{\mathbb{R}^d} \phi_{8\varepsilon}(w') \frac{1}{\|t - s\|^{d/2}} \exp\left(-\frac{\|w' + y - C_{s, t}x\|^2}
{2 \,c_{\ref{5.9}}'\|t - s\|}\right)  dw' \right] \mu(ds\,dx)\, \mu(dt\,dy)
\end{align}
for some constant $c_{\ref{7.12}} > 0$ independent of $\varepsilon$.
By the triangle inequality and \eqref{5.8}, there is a constant $c_{\ref{7.13}} > 0$ such that for all $x\in F$,
\begin{align}
\begin{aligned}\label{7.13}
& \|w' + y - x\|^2 \\
& \leq \Big(\|w' + y - C_{s, t}x\| + \|(1 - C_{s, t})x\|\Big)^2\\
&  \leq \Big(\|w' + y - C_{s, t}x\| + c_{\ref{7.13}} \|t - s\|^{1/2}\Big)^2\\
&  = \|w' + y - C_{s, t}x\|^2 + c_{\ref{7.13}}^2 \|t - s\| + 2 \,c_{\ref{7.13}}\,\|w' + y - C_{s, t}x\| \,\|t - s\|^{1/2}.
\end{aligned}
\end{align}
Now, consider two cases.
If (i) $\|w' + y - C_{s, t}x\| > \|t - s\|^{1/2}$, then
\begin{align}
\begin{aligned}\label{7.14}
\exp\left(- \frac{\|w' + y - C_{s, t}x\|^2}{2 \,c_{\ref{5.9}}'\|t - s\|} \right) 
& = \exp\left(- \frac{(1 + c_{\ref{7.13}}^2 + 2\,c_{\ref{7.13}})\|w' + y - C_{s, t}x\|^2}{2 (1 + c_{\ref{7.13}}^2 + 2\,c_{\ref{7.13}})\,
c_{\ref{5.9}}'\|t - s\|} \right)  \\
& \leq \exp\left(- \frac{\|w' + y - x\|^2}{2\, c_{\ref{7.14}}\|t - s\|} \right),
\end{aligned}
\end{align}
where $c_{\ref{7.14}} = (1 + c_{\ref{7.13}}^2 + 2\,c_{\ref{7.13}}) \, c_{\ref{5.9}}'$.
Note that we have $c_{\ref{7.14}} > c_{\ref{5.9}}' > 1$.
On the other hand, if (ii) $\|w' + y - C_{s, t}x\| \leq \|t - s\|^{1/2}$, then
\begin{align}
\notag
\exp\left(- \frac{\|w' + y - C_{s, t}x\|^2}{2 \,c_{\ref{5.9}}'\|t - s\|} \right) 
& \leq \exp\left(- \frac{\|w' + y - x\|^2}{2 \,c_{\ref{5.9}}'\|t - s\|}\right) \exp\left( - \frac{(c_{\ref{7.13}}^2 + 2 \,c_{\ref{7.13}})\|t-s\|}
{2\,c_{\ref{5.9}}' \|t-s\|} \right)\\
& \leq \exp\left(- \frac{\|w' + y - x\|^2}{2 \,c_{\ref{7.14}}\|t - s\|} \right).
\label{7.15}
\end{align}
It follows from \eqref{7.12} and the above that
\begin{align}
\begin{aligned}
\label{5.17}
\E[Z_\varepsilon(\mu)^2] & \leq c_{\ref{7.12}} \iint \left[\int_{\mathbb{R}^d}\phi_{8\varepsilon}(w') \, r_{s, t}(w' + y - x) \, dw' \right]
\mu(ds\,dx)\, \mu(dt\,dy)\\
& = c_{\ref{7.12}} \iint (\phi_{8\varepsilon} \ast r_{s, t})(x - y) \,\mu(ds\,dx)\, \mu(dt\,dy),
\end{aligned}
\end{align}
where
\begin{align}\label{7.16}
r_{s, t}(w) = \frac{1}{\|t - s\|^{d/2}} \exp\left(-\frac{\|w\|^2}{2c_{\ref{7.14}}\|t - s\|}\right), \quad w \in \mathbb{R}^d.
\end{align}
Now, we may apply Lemma \ref{lemma4.2} followed by Lemma \ref{lemma4.1} and the fact that $\widehat{\phi}_{8\varepsilon} = |\widehat{f}_{8\varepsilon}|^2 \le 1$ in order to obtain
\begin{align} 
\begin{aligned}\label{7.18}
\E[Z_\varepsilon(\mu)^2] & \leq c_{\ref{7.18}} \iint (\phi_{8\varepsilon} \ast p_{s, t})(x - y) \, \mu(ds\,dx)\, \mu(dt\,dy)\\
& \le c_{\ref{7.18}} \iint p_{s, t}(x - y) \, \mu(ds\,dx)\, \mu(dt\,dy)
\end{aligned}
\end{align}
for some constant $c_{\ref{7.18}} > 0$ independent of $\varepsilon$.
On the other hand,
\begin{align}\label{7.19}
\notag
\E(Z_\varepsilon(\mu))
& =\int \left[\int_{\|w\| \leq b + 2} \phi_{\varepsilon}(w - x) \frac{1}{(2\pi \prod_{i=1}^N s_i)^{d/2}} \exp\left(- \frac{\|w\|^2}{2 \prod_{i=1}^N s_i} \right) \, dw \right]
\mu(ds\, dx)\\
\notag
& \geq \frac{1}{(2\pi b^N)^{d/2}} \exp\left( - \frac{(b + 2)^2}{2 a^N}\right) \int \left[\int_{\|w\| \leq b+ 2} \phi_\varepsilon(w - x)\, dw \right] 
\mu(ds\, dx)\\
& = c_{\ref{7.19}} > 0,
\end{align}
where $c_{\ref{7.19}}$ is a constant independent of $\varepsilon$.
Hence from \eqref{3.16}, we have
\begin{align}
\begin{aligned} 
\P(W(E) \cap F^\varepsilon \neq \varnothing) &\geq c_{\ref{7.19}}^{\,2}\, c_{\ref{7.18}}^{-1} \left[\iint p_{s, t}(x - y) \, \mu(ds\,dx)\, \mu(dt\,dy)\right]^{-1}\\
& = c_{\ref{7.19}}^{\,2}\, c_{\ref{7.18}}^{-1}\, \mathscr{E}_0(\mu)^{-1}.
\end{aligned}
\end{align}
Letting $\varepsilon \downarrow 0$ yields
\begin{align}
\begin{aligned}
\P(W(E) \cap F \neq \varnothing) &\geq c_{\ref{7.19}}^{\,2}\, c_{\ref{7.18}}^{-1}\, \mathscr{E}_0(\mu)^{-1}.
\end{aligned}
\end{align}
Hence, we may take supremum over all $\mu \in P_0(E\times F)$ to obtain the lower bound in \eqref{hp:bd}.

\medskip

\noindent
\textit{Part 2: Upper bound.}
Clearly, we may assume that $\P(W(E) \cap F \neq \varnothing) > 0$ (otherwise, the upper bound in \eqref{hp:bd} holds trivially).
With the convention $\inf \varnothing = \infty$, define
\begin{align}
\tau_1 = \inf\{ t_1 : \exists\, t_2, \dots, t_N \text{ such that }  t=(t_1, \dots, t_N) \in E \text{ and } W(t) \in F \}.
\end{align}
If $\tau_1 = \infty$, then define $\tau_i = \infty$ for all $2 \leq i \leq N$. If $\tau_{1} \neq \infty$, then define 
\begin{align}\begin{split}
\tau_{i} = \inf\{ t_i : \, &\exists \, t_{i+1},\dots, t_N \text{ such that }\\
&(\tau_1, \dots, \tau_{i-1}, t_i, \dots, t_N) \in E \text{ and } W(\tau_1, \dots, \tau_{i-1}, t_i, \dots, t_N) \in F \}
\end{split}\end{align}
 inductively  for $i = 2, \dots, N$.
Let $\tau = (\tau_1, \dots, \tau_N)$.
Note that $W(\tau) \in F$ on the event $G:=\{\tau_i \neq \infty \text{ for all } i \}$.
Since $\P(G) = \P(W(E) \cap F \neq \varnothing) > 0$, we can define a Borel probability measure $\mu$ on $E\times F$ by 
\begin{align}\label{D:mu}
\mu(A) = \P\{ (\tau, W(\tau)) \in A\mid G\}.
\end{align}
Since $F$ has Lebesgue measure zero, $\mu(\{t\} \times F) = 0$ for all $t \in E$, i.e., $\mu\in P_0(E\times F)$.
Our goal is to show that $\P(W(E) \cap F \ne \varnothing) \le K \mathscr{E}_0(\mu)^{-1}$ for some $K>1$ that depends only on $(N,d,a,b)$.
To this end, let 
\begin{align}
&\phi_\varepsilon(w) = \frac{1}{(2\pi\varepsilon^2)^{d/2}} \exp\left(- \frac{\|w\|^2}{2\varepsilon^2}\right), \quad w \in \mathbb{R}^d,\\
&Z_\varepsilon(\mu) = \int_{E\times F} \phi_\varepsilon(W(t) - y) \, \mu(dt\,dy).
\end{align}
Consider an arbitrary $s \in [a, b]^N$ and $\pi \subset \{1, \dots, N\}$.
We have
\begin{align}
\begin{aligned}\label{7.20}
\E[Z_\varepsilon(\mu)|\mathscr{F}_\pi(s)]
& = \int \E[\phi_\varepsilon(W(t) - y)| \mathscr{F}_\pi(s)] \,\mu(dt\,dy)\\
& \geq \int_{\substack{t \succcurlyeq_\pi s,\\t \neq s \,\,\,\,}} \E[\phi_{\varepsilon}(W_s(t) + C_{s, t}W(s) - y)|\mathscr{F}_\pi(s)] \, 
\mu(dt\, dy).
\end{aligned}
\end{align}
Let 
\begin{align}
\tilde{q}_{s, t}(w) = \frac{1}{\|t - s\|^{d/2}} \exp \left( - \frac{\|w\|^2}{2 \,c_{\ref{5.9}}\|t - s\|} \right), \quad w \in \mathbb{R}^d.
\end{align}
By the independence of $\{W_s(t) : t \succcurlyeq_\pi s\}$ and $\mathscr{F}_\pi(s)$, and \eqref{5.9},
\begin{align}
\begin{aligned}\label{5.18}
&\E[Z_\varepsilon(\mu)|\mathscr{F}_\pi(s)]\\
& \geq c_{\ref{5.18}} \int_{\substack{t \succcurlyeq_\pi s\\t \neq s}} \left[ \int_{\mathbb{R}^d} \phi_{\varepsilon}(w - (y - C_{s, t}W(s)))\, \tilde{q}_{s, t}(w)\, dw  \right] \mu(dt\,dy)\, {\bf 1}_{\{\|W(s)\| \leq b\}}\\
& = c_{\ref{5.18}} \int_{\substack{t \succcurlyeq_\pi s\\t \neq s}} \left[ \int_{\mathbb{R}^d} \phi_{\varepsilon}(w')\, \tilde{q}_{s, t}(w' + y - C_{s, t}W(s))\, dw' \right] \mu(dt\,dy)\, {\bf 1}_{\{\|W(s)\| \leq b\}}.
\end{aligned}
\end{align}
By the triangle inequality and \eqref{5.8}, we can find constants $c_{\ref{7.28}}, c_{\ref{7.28}}' > 0$ such that
\begin{align}
\begin{aligned}\label{7.28}
& \quad \|w' + y - C_{s, t}W(s)\|^2 \\
& \leq \|w' + y - W(s)\|^2 + \|(1 - C_{s, t})W(s)\|^2 + 2 \|w' + y - W(s)\| \|(1 - C_{s, t})W(s)\|\\
&\leq \|w' + y - W(s)\|^2 + c_{\ref{7.28}} \|t - s\| + c_{\ref{7.28}}' \|w' + y - W(s)\| \|t - s\|^{1/2}.
\end{aligned}
\end{align}
Consider the following two cases. If (i) $\|w' + y - W(s)\| \leq \|t - s\|^{1/2}$, then
\begin{align}\label{E:3.35}
\notag
\exp\left( -\frac{\|w' + y - C_{s, t} W(s)\|^2}{2 \,c_{\ref{5.9}} \|t - s\|} \right)
&\geq \exp\left(- \frac{\|w' + y - W(s)\|^2}{2 \, c_{\ref{5.9}}\|t - s\|} - \frac{(c_{\ref{7.28}}+ c_{\ref{7.28}}')\|t-s\|}{2 \,c_{\ref{5.9}}\|t-s\|}\right)\\
& \geq c_{\ref{E:3.35}}\exp\left(- \frac{\|w' + y - W(s)\|^2}{2 \, c_{\ref{5.9}}\|t - s\|}\right).
\end{align}
If (ii) $\|w' + y - W(s)\| > \|t - s\|^{1/2}$, then
\begin{align}\label{5.21}
\exp\left( -\frac{\|w' + y - C_{s, t} W(s)\|^2}{2 \,c_{\ref{5.9}}\|t - s\|}\right)
 \geq \exp\left(- \frac{(1 + c_{\ref{7.28}} + c_{\ref{7.28}}')\|w' + y - W(s)\|^2}{2 \,c_{\ref{5.9}}\|t - s\|}\right).
\end{align}
Hence, we have
\begin{align}
\begin{aligned}\label{5.22}
&\sup_{s \in [a, b]^N \cap \Q^N}\E[Z_\varepsilon(\mu)|\mathscr{F}_\pi(s)] \\
& \geq c_{\ref{5.22}} \int_{\substack{t \succcurlyeq_\pi s\\t \neq s}} \left[\int_{\mathbb{R}^d}\phi_{\varepsilon}(w') \, q_{s, t}(w' + y - W(s)) 
\, dw' \right] \mu(dt\, dy) \, {\bf 1}_{\{\|W(s)\| \leq b\}}\\
& = c_{\ref{5.22}} \int_{\substack{t \succcurlyeq_\pi s\\t \neq s}} (\phi_{\varepsilon} \ast q_{s, t})(W(s) - y) \, \mu(dt\, dy) \, 
{\bf 1}_{\{\|W(s)\| \leq b\}}
\end{aligned}
\end{align}
uniformly for all $s\in [a, b]^N \cap \Q^N$ and $\pi \subset \{1, \dots, N\}$, where
\begin{align}\label{7.32}
q_{s, t}(w) = \frac{1}{\|t - s\|^{d/2}} \exp\left(- \frac{\|w\|^2}{2 c_{\ref{7.32}}\|t - s\|} \right)
\end{align}
with $c_{\ref{7.32}} = c_{\ref{5.9}}/(1 + c_{\ref{7.28}} + c_{\ref{7.28}}') < 1$ since $c_{\ref{5.9}} < 1$.

On the event $G$, we have $\tau \in E \subset [a, b]^N$ and $\|W(\tau)\| < b$.
Also, $s \mapsto (\phi_\varepsilon \ast q_{s, t})(W(s) - y)$ is continuous. By choosing a random sequence $(s_m)$ in 
$[a, b]^N \cap \mathbb{Q}^N$ converging to $\tau$, and by applying \eqref{5.22} and Fatou's lemma, we obtain
\begin{align}
\begin{aligned}
&\sup_{s \in [a, b]^N \cap \,\mathbb{Q}^N} \E[Z_\varepsilon(\mu)|\mathscr{F}_\pi(s)] \\
& \geq c_{\ref{5.22}} \liminf_{m\to \infty} \int_{\substack{t \succcurlyeq_\pi s_m\\ t \neq s_m}} (\phi_{\varepsilon} \ast q_{s_m, t})(W(s_m) - y) 
\,\mu(dt\,dy)\,{\bf 1}_{\{\|W(s_m)\|\le b\}}\\
& \geq c_{\ref{5.22}} 
 \int_{\substack{t \succcurlyeq_\pi \tau\\ t \neq \tau}} (\phi_{\varepsilon} \ast q_{\tau, t})(W(\tau) - y) \,\mu(dt\,dy)\,{\bf 1}_{G}.
\end{aligned}
\end{align}
Squaring both sides, taking expectations, and then using the Cauchy--Schwarz inequality, we get that
\begin{align}
\begin{aligned}\label{7.34}
& \E\left(\sup_{s \in [a, b]^N \cap \,\mathbb{Q}^N} \E[Z_{\varepsilon}(\mu)|\mathscr{F}_\pi(s)]^2\right) \\
&\geq  c_{\ref{5.22}}\, \P(G) \int \left(\int_{\substack{t \succcurlyeq_\pi s\\ t\neq s}} (\phi_{\varepsilon} \ast q_{s, t})(x - y) \,
\mu(dt\,dy)\right)^2 \mu(ds\,dx)\\
& \geq c_{\ref{5.22}} \, \P(G) \left(\iint_{t \succcurlyeq_\pi s} (\phi_{\varepsilon} \ast q_{s, t})(x - y) \,\mu(dt\,dy)\, \mu(ds\,dx)\right)^2\\
& = c_{\ref{5.22}}\, \P(G)\, J_{\varepsilon, \pi}^{\, 2},
\end{aligned}
\end{align}
where
\begin{align}
J_{{\varepsilon}, \pi} = \iint_{t \succcurlyeq_\pi s} (\phi_{\varepsilon} \ast q_{s, t})(x - y) \,\mu(ds\,dx)\, \mu(dt\,dy).
\end{align}

Next, we establish an upper bound for the expectation in \eqref{7.34}. Using Cairoli's maximal inequality \eqref{E:Cairoli}, we have
\begin{align}\label{7.36}
\notag
& \E\left( \sup_{s \in [a, b]^N \cap\, \mathbb{Q}^N} \E[Z_\varepsilon(\mu)|\mathscr{F}_\pi(s)]^2\right) 
 \leq 4^N \E[Z_\varepsilon(\mu)^2]\\
& = 4^N \iint \E[\phi_\varepsilon(W(s) - x)\phi_\varepsilon(W(t) - y)] \, \mu(ds\, dx) \, \mu(dt\, dy)\\
& = 4^N \sum_{\pi \subset \{1,\dots,N\}} \iint_{t \succcurlyeq_\pi s} \E[\phi_\varepsilon(W(s) - x)\phi_\varepsilon(W_s(t) + C_{s, t}W(s) -y)] \,
 \mu(ds\, dx) \, \mu(dt\, dy).
\notag
\end{align}
We claim that 
\begin{align}
\begin{aligned}\label{7.37}
& \E[\phi_\varepsilon(W(s) - x)\phi_\varepsilon(W_s(t) + C_{s, t}W(s) -y)] \\
& \quad \quad \leq  c_{\ref{7.37}}\, \E[ \phi_{2 \varepsilon}(W(s) - x)\phi_{2\varepsilon}(W_s(t) + C_{s, t}x - y)]
\end{aligned}
\end{align}
for some constant $c_{\ref{7.37}} > 0$ independent of $\varepsilon$. 
It suffices to show that
\begin{align}
\begin{aligned}
& \exp\left(-\frac{\|W(s) - x\|^2}{2 \varepsilon^2} \right) \exp \left(- \frac{\|W_s(t) + C_{s, t} W(s) - y\|^2}{2 \varepsilon^2}\right) \\
& \quad \leq \exp\left(- \frac{\|W(s) - x\|^2}{2(2 \varepsilon)^2}\right) \exp\left(- \frac{\|W_s(t) + C_{s, t}x - y\|^2}{2(2\varepsilon)^2} \right) .
\end{aligned}
\end{align}
To prove this, note that by the triangle inequality, we have
\begin{align}
\begin{aligned}
\|W_s(t) + C_{s, t}x - y\|^2 \leq \|W_s(t) + & C_{s, t}W(s) - y\|^2 + \|C_{s, t}(W(s) - x)\|^2\\ 
& + 2\|W_s(t) + C_{s, t}W(s) - y\|\|C_{s, t}(W(s) - x)\|.
\end{aligned}
\end{align}
Now consider two cases.
If (i) $\|W_s(t) + C_{s, t}W(s) - y\| \leq \|C_{s, t}(W(s) - x)\|$, then 
\begin{align}
\|W_s(t) + C_{s, t}x - y\|^2 \leq \|W_s(t) + C_{s, t} W(s) - y\|^2 + 3 \|C_{s, t}(W(s) - x)\|^2
\end{align}
and it follows that
\small
\begin{align}
\begin{aligned}
& \quad \exp\left(-\frac{\|W(s) - x\|^2}{2 \varepsilon^2} \right) \exp\left(-\frac{\|W_s(t) + C_{s, t} W(s) - y\|^2}{2 \varepsilon^2}\right) \\
& \leq \exp\left(-\frac{\|W(s) - x\|^2}{8 \varepsilon^2} - \frac{3\|C_{s, t}(W(s) - x)\|^2}{8 \varepsilon^2} \right) \exp\left(-\frac{\|W_s(t) 
+ C_{s, t} W(s) - y\|^2}{8 \varepsilon^2}\right) \\
& \leq \exp\left(- \frac{\|W(s) - x\|^2}{2(2 \varepsilon)^2}\right) \exp\left(- \frac{\|W_s(t) + C_{s, t} x - y\|^2}{2(2\varepsilon)^2} \right) .
\end{aligned}
\end{align}
\normalsize
If (ii) $\|W_s(t) + C_{s, t}W(s) - y\| > \|C_{s, t}(W(s) - x)\|$, then
\begin{align}
\|W_s(t) + C_{s, t}x - y\|^2 \leq 3 \|W_s(t) + C_{s, t} W(s) - y\|^2 + \|C_{s, t}(W(s) - x)\|^2,
\end{align}
thus we have
\small
\begin{align}
\begin{aligned}
& \quad \exp\left(-\frac{\|W(s) - x\|^2}{2 \varepsilon^2} \right) \exp\left(-\frac{\|W_s(t) + C_{s, t} W(s) - y\|^2}{2 \varepsilon^2}\right) \\
& \leq \exp\left(-\frac{\|W(s) - x\|^2}{4 \varepsilon^2} - \frac{3\|C_{s, t}(W(s) - x)\|^2}{8 \varepsilon^2} \right) \exp\left(-\frac{3\|W_s(t) 
+ C_{s, t} W(s) - y\|^2}{8 \varepsilon^2}\right) \\
& \leq \exp\left(- \frac{\|W(s) - x\|^2}{2(2 \varepsilon)^2}\right) \exp\left(- \frac{\|W_s(t) + C_{s, t}x - y\|^2}{2(2\varepsilon)^2} \right) .
\end{aligned}
\end{align}
\normalsize
Hence the claim \eqref{7.37} is proved.
Using \eqref{7.37}, independence of $\{W_s(t) : t \succcurlyeq_\pi s\}$ and $W(s)$, the bound
\begin{align}
\E[\phi_{2\varepsilon}(W(s) - x)] \leq {(2\pi a^N)^{-d/2}},
\end{align}
and \eqref{5.9}, we see that \eqref{7.36} is further bounded above by
\begin{align}
\begin{aligned}\label{7.45}
& \quad 4^N c_{\ref{7.37}} \sum_\pi \iint_{t \succcurlyeq_\pi s} \E[\phi_{2\varepsilon}(W(s) - x)]  \E[\phi_{2\varepsilon}(W_s(t) 
+ C_{s, t}x - y)]  \, \mu(ds\, dx) \, \mu(dt\, dy)\\
& \leq c_{\ref{7.45}} \iint \int_{\mathbb{R}^d} \phi_{2\varepsilon}(w + C_{s, t}x - y) \frac{1}{\|t - s\|^{d/2}} \exp\left(- \frac{\|w\|^2}
{2 \,c_{\ref{5.9}}' \|t - s\|} \right) \,dw\, \mu(ds\, dx) \, \mu(dt\, dy)\\
& = c_{\ref{7.45}} \iint \int_{\mathbb{R}^d} \phi_{2\varepsilon}(w') \frac{1}{\|t - s\|^{d/2}} \exp\left(- \frac{\|w' + y - C_{s, t} x\|^2}
{2 \,c_{\ref{5.9}}' \|t - s\|} \right) \,dw'\, \mu(ds\, dx) \, \mu(dt\, dy).
\end{aligned}
\end{align}
Now we may repeat the argument from \eqref{7.13} to \eqref{7.15} to deduce that
\begin{align}
\exp\left(- \frac{\|w' + y - C_{s, t}x\|^2}{2\, c_{\ref{5.9}}'\|t - s\|} \right)  \leq \exp\left(- \frac{\|w' + y - x\|^2}{2 \,c_{\ref{7.14}}\|t - s\|} \right)
\end{align}
for some constant $c_{\ref{7.14}} > 1$.
Let $r_{s, t}$ be defined as in \eqref{7.16}.
Then from \eqref{7.36} and \eqref{7.45},
\begin{align}
\begin{aligned}\label{7.50}
& \quad \E\left(\sup_{s \in [a, b]^N \cap\, \mathbb{Q}^N} \E[Z_\varepsilon(\mu)|\mathscr{F}_\pi(s)]^2 \right)\\
& \leq c_{\ref{7.45}} \iint \int_{\mathbb{R}^d} \phi_{2\varepsilon}(w') \, r_{s, t}(w' + y - x) \, dw' \, \mu(ds\, dx) \, \mu(dt\, dy)\\
& = c_{\ref{7.45}} \iint (\phi_{2\varepsilon} \ast r_{s, t})(x - y) \, \mu(ds\, dx) \, \mu(dt\, dy)\\
& \leq c_{\ref{7.50}} \iint (\phi_\varepsilon \ast p_{s, t})(x - y)\, \mu(ds\, dx) \, \mu(dt\, dy) := c_{\ref{7.50}} \, I_\varepsilon,
\end{aligned}
\end{align}
where we use Lemma \ref{lemma4.2} in the last inequality and we have set
\begin{align}
I_\varepsilon = \iint (\phi_\varepsilon \ast p_{s, t})(x - y)\, \mu(ds\, dx) \, \mu(dt\, dy).
\end{align}

Now, combine \eqref{7.34} and \eqref{7.50}, and sum over all $\pi \subset \{1, \dots, N\}$ to get
\begin{align}
\begin{aligned}\label{5.29}
2^N c_{\ref{7.50}} \, I_\varepsilon
& \geq c_{\ref{5.22}} \, \P(G)\sum_\pi J_{\varepsilon, \pi}^2\\
& \geq 2^{-N}c_{\ref{5.22}}\,\P(G) \bigg(\sum_\pi J_{\varepsilon, \pi}\bigg)^2\\
& = 2^{-N}c_{\ref{5.22}}\, \P(G) \left(\iint (\phi_{\varepsilon} \ast q_{s, t})(x - y) \, \mu(ds\, dx) \, \mu(dt\, dy)  \right)^2.
\end{aligned}
\end{align}
Using Lemma \ref{lemma4.2} again, we have
\begin{align}
\begin{aligned}\label{7.52}
c_{\ref{7.52}}\, I_\varepsilon & \geq \P(G) \left( \iint (\phi_{\varepsilon} \ast p_{s, t})(x - y)  \, \mu(ds\, dx) \, \mu(dt\, dy) \right)^2\\
& = \P(G) \, I_\varepsilon^2,
\end{aligned}
\end{align}
for some constant $c_{\ref{7.52}} > 0$ independent of $\varepsilon$.
Note that $I_\varepsilon$ is finite for each $\varepsilon \in (0, 1)$.
Therefore,
\begin{align}
I_\varepsilon = \iint (\phi_\varepsilon \ast p_{s, t})(x - y)\, \mu(ds\, dx) \, \mu(dt\, dy) \leq c_{\ref{7.52}} \, \P(G)^{-1}.
\end{align}
Since $\P(G) = \P(W(E) \cap F \neq \varnothing) > 0$, letting $\varepsilon \to 0$ yields 
\begin{align}
\P(W(E) \cap F \neq \varnothing) \le c_{\ref{7.52}} \, \mathscr{E}_0(\mu)^{-1} \le c_{\ref{7.52}} \, C_0(E\times F).
\end{align}
This completes the proof of Theorem \ref{T:hit}.
\end{proof}

\section{Intersection with the range of an additive stable process}

The proof of Theorem \ref{T:dim:cap}, which will be given in the next section, is based on a codimension 
argument where we check whether the set $W(E) \cap F$ intersects the range of an additive stable process $X$.
The goal of this section is to prove Proposition \ref{thm6.1} below, which gives a necessary and sufficient condition for $W(E) \cap F$ to intersect the range of $X$.

For the rest of the paper, we assume that 
$\{X^{(1)}(u) : u \ge 0\}, \dots, \{ X^{(n)}(u) : u \ge 0 \}$
are $\mathbb{R}^d$-valued $\alpha$-stable processes with $0<\alpha<2$ and 
\[ \E(e^{i \langle \xi, X^{(k)}(u)\rangle}) = e^{-u \|\xi\|^\alpha/2},
\]
and assume that $W, X^{(1)}, \dots, X^{(n)}$ are independent. 
Let $X = \{X(u): u \in \R^n_+\}$ be the $\mathbb R^d$-valued additive $\alpha$-stable process defined by 
\[
X(u) = X^{(1)}(u_1) + \cdots + X^{(n)}(u_n)
\]
for any $u = (u_1, \dots, u_n) \in \mathbb{R}^n_+$. 
The aforementioned codimension argument is based on the following result on the hitting probability for the range of $X$.

\begin{proposition}\cite[Theorem 4.4]{KX2005}
\label{prop4.1}
Let $X$ be the additive $\alpha$-stable process defined above. Then for any Borel set $G \subset \mathbb{R}^d$, 
$\P(G \cap X(\mathbb{R}^n_+ \setminus \{0\}) \ne \varnothing) > 0$ if and only if ${\EuScript Cap}_{d - \alpha n}(G) > 0$,
where ${\EuScript Cap}_{d - \alpha n}$ denotes the $(d - \alpha n)$-dimensional Bessel-Riesz capacity in the Euclidean 
metric $($cf. e.g., \cite{Fal, K,Mattila}$)$.
\end{proposition}

The next lemma establishes the commuting property of the filtration generated by the Brownian sheet and the additive 
stable process, and Cairoli's maximal inequality for this filtration.

\begin{lemma}\label{lemma6.2}
Let $\pi \subset \{1, \dots, N\}$ and $\preccurlyeq^\pi$ be the partial ordering on $\mathbb{R}^N_+ \times \mathbb{R}^n_+$ defined 
by $(s, u) \preccurlyeq^\pi (t, v)$ iff $s \preccurlyeq_\pi t$ and $u \preccurlyeq v$.
Let $\mathscr{F}_s^\pi = \sigma\{W(t) : t \preccurlyeq_\pi s\}$, $\mathscr{G}_u = \bigvee_{i=1}^n \mathscr{G}^{(i)}_{u_i}$, where 
$\mathscr{G}^{(i)}_{u_i} = \sigma\{X^{(i)}(t) : t \leq u_i \}$, and $\mathscr{H}_{s, u}^\pi = \mathscr{F}_s^\pi \vee \mathscr{G}_u$.
Then $\{ \mathscr{H}_{s, u}^\pi : (s, u) \in \mathbb{R}^N_+ \times \mathbb{R}^n_+ \}$ is a commuting filtration with respect to the ordering 
$\preccurlyeq^\pi$.  As a consequence, for any $Z \in L^2(\P)$,
the following Cairoli maximal inequality holds:
\begin{align}
\E\left(\sup_{(s, u) \in \mathbb{Q}^N_+ \times \mathbb{Q}^n_+} [\E(Z|\mathscr{H}_{s, u}^{\pi})]^2 \right) \leq 4^{N+n} \E(Z^2).
\end{align}
\end{lemma}

\begin{proof}
First, we know that $\{\mathscr{F}_s^\pi\}$ is commuting with respect to $\preccurlyeq_\pi$ thanks to \cite[Proposition 3.1]{KX2007}.
To see that $\{\mathscr{G}_u\}$ is commuting with respect to $\preccurlyeq$, it suffices to show that $\E(Y|\mathscr{G}_u) = 
\E(Y|\mathscr{G}_{u \wedge v})$ for $Y$ of the form $\prod_{i=1}^n {\bf 1}_{A_i}$, where $A_i \in \mathscr{G}^{(i)}_{v_i}$. 
By the independence of the $X^{(i)}$'s, we have $\E(Y|\mathscr{G}_u) = \prod_{i=1}^n \E({\bf 1}_{A_i}|\mathscr{G}^{(i)}_{u_i})$ and 
$\E(Y|\mathscr{G}_{u \wedge v}) = \prod_{i=1}^n \E({\bf 1}_{A_i}|\mathscr{G}^{(i)}_{u_i \wedge v_i})$. Since $A_i \in \mathscr{G}_{v_i}^{(i)}$, we have $\E({\bf 1}_{A_i}|\mathscr{G}^{(i)}_{u_i}) = \E({\bf 1}_{A_i}|\mathscr{G}^{(i)}_{u_i \wedge v_i})$, so $\{\mathscr{G}_u\}$ is commuting.
Then by the independence of $W$ and the $X^{(i)}$'s, for any $A \in \mathscr{F}_t^\pi$ and $B \in \mathscr{G}_v$, we have 
\[
\E({\bf 1}_{A\cap B}|\mathscr{H}_{s, u}^\pi) = \E({\bf 1}_A|\mathscr{F}_s^\pi) \E({\bf 1}_B|\mathscr{G}_u) = \E({\bf 1}_A|\mathscr{F}^\pi_{s \wedge_\pi t}) 
\E({\bf 1}_B|\mathscr{G}_{u \wedge v}) = \E({\bf 1}_{A \cap B} |\mathscr{H}^\pi_{s \wedge_\pi t, u \wedge v}).\]
This implies that $\E(Y|\mathscr{H}_{s, u}^\pi) = \E(Y|\mathscr{H}_{s \wedge_\pi t, u \wedge v}^\pi)$ for all bounded 
$\mathscr{H}_{t, v}^\pi$-measurable function $Y$. Hence $\{\mathscr{H}^\pi_{s, u}\}$ is commuting with respect to $\preccurlyeq^\pi$.

To prove the maximal inequality, let $\{q_0, q_1, \dots\}$ be an enumeration of $\mathbb{Q}_+$ and $\mathbb{Q}_m = \{q_0, \dots, q_m\}$.
We can apply Cairoli's maximal inequality \cite[Theorem 2.3.1, p.19; Theorem 3.5.1, p.37]{K} to the martingale $\E[Z|\mathscr{H}^\pi_{s, u}]
$ with the order $\preccurlyeq^\pi$ on each $\mathbb{Q}_m^N \times \mathbb{Q}_m^n$ because $(\mathbb{Q}^N_m \times \mathbb{Q}^n_m, \preccurlyeq^\pi)$ is order isomorphic to $\{ (0, \dots, 0) \preccurlyeq s \preccurlyeq (m, \dots, m) : s \in \mathbb{N}_0^{N+n}\}$. 
Then let $m \to \infty$ to finish the proof.
\end{proof}

\begin{lemma}\label{lemma6.3}
Let $p_{s, t}, r_{s, t}$ and $q_{s, t}$ be as in Lemma \ref{lemma4.2}.
Let $E \subset \mathbb{R}^N$, $F \subset \mathbb{R}^d$ be compact sets and $\mu \in P_0(E\times F)$.
Let $\kappa \in L^1(\mathbb{R}^d)$ be a positive definite lower semicontinuous function. 
Let $\phi_1, \phi_2$ be positive definite functions in either $C_c(\mathbb{R}^d)$ or the Schwartz space $\mathcal{S}(\mathbb{R}^d)$ such that $0 \leq \widehat{\phi}_1 \leq \widehat{\phi}_2$.
Assume that $\phi_1 \ast \kappa$ 
and $\phi_2 \ast \kappa$ are continuous.
Then
\begin{align}
\begin{aligned}\label{6.1}
&\iint (\phi_1 \ast p_{s, t})(x - y)\, (\phi_1 \ast \kappa)(x - y) \, \mu(ds\, dx) \, \mu(dt\, dy) \\
& \hspace{50pt} \leq \iint (\phi_2 \ast p_{s, t})(x - y)\,(\phi_2 \ast \kappa)(x - y)\, \mu(ds\, dx) \, \mu(dt\, dy)
\end{aligned}
\end{align}
and
\begin{align}
\begin{aligned}\label{6.2}
&c_2^{-(1+\frac{d}{2})} \iint (\phi_1 \ast r_{s, t})(x - y)\, (\phi_1 \ast \kappa)(x - y) \, \mu(ds\, dx) \, \mu(dt\, dy)  \\
& \hspace{50pt} \leq \iint (\phi_1 \ast p_{s, t})(x - y) \, \mu(ds\, dx) \, (\phi_1 \ast \kappa)(x - y) \,\mu(dt\, dy) \\
&\hspace{80pt} \leq c_1^{-(1+\frac{d}{2})} \iint (\phi_1 \ast q_{s, t})(x - y) \,(\phi_1 \ast \kappa)(x - y) \, \mu(ds\, dx) \, \mu(dt\, dy).
\end{aligned}
\end{align}
\end{lemma}

\begin{proof}
As in Lemma \ref{lemma4.2}, it suffices to prove that for $\nu = \sum_{i=1}^n a_i (\delta_{s_i} \otimes \delta_{x_i})$,
\begin{align}
\begin{aligned}\label{6.3}
&\iint [(\phi_1 \ast p_{s, t})\cdot (\phi_1 \ast \kappa)](x - y) \, \nu(ds\, dx) \, \nu(dt\, dy)\\
& \quad \quad \leq \iint [(\phi_2 \ast p_{s, t})\cdot(\phi_2 \ast \kappa)](x - y)\, \nu(ds\, dx) \, \nu(dt\, dy).
\end{aligned}
\end{align}
The Fourier transform of $(\phi_k \ast p_{s, t}) \cdot (\phi_k \ast \kappa)$ is equal to 
$(\widehat{\phi}_k \,\widehat{p}_{s, t}) \ast (\widehat{\phi}_k\, \widehat{\kappa})$.
By Lemma \ref{lemma4.1}, 
\begin{align*}
& \quad \iint (\phi_k \ast p_{s, t})(x - y)\, (\phi_k \ast \kappa)(x - y) \, \nu(ds\, dx) \, \nu(dt\, dy)\\
& = (2\pi)^{-d} \sum_{i, j =1}^n a_i a_j \iint \delta_{s_i}(ds) \delta_{s_j}(dt) \int d\xi \int dz \, \widehat{\phi}_k(z)  \widehat{p}_{s, t}(z)  \widehat{\phi}_k(\xi - z) \widehat{\kappa}(\xi - z)  \widehat{\delta}_{x_i}(\xi) \overline{\widehat{\delta}_{x_j}(\xi)}\\
& = (2\pi)^{-\frac{d}{2}} \sum_{i, j = 1}^n a_i a_j \int dz \int d\xi \,\widehat{\phi}_k(z) \widehat{\phi}_k(\xi - z) \widehat{\kappa}(\xi - z) \widehat{\delta}_{x_i}(\xi) \overline{\widehat{\delta}_{x_j}(\xi)}  \iint \delta_{s_i}(ds) \delta_{s_j}(dt) P_z(s - t)\\
& = (2\pi)^{-N - \frac{d}{2}} \sum_{i, j = 1}^n a_i a_j \int dz \int d\xi \,\widehat{\phi}_k(z) \widehat{\phi}_k(\xi - z) \widehat{\kappa}(\xi - z) \widehat{\delta}_{x_i}(\xi)  \overline{\widehat{\delta}_{x_j}(\xi)} \int d\eta\, \widehat{P}_z(\eta)  \widehat{\delta}_{s_i}(\eta)  \overline{\widehat{\delta}_{s_j}(\eta)} \\
& = (2\pi)^{-N - \frac{d}{2}} \int dz \iint d\xi \, d\eta\, \widehat{P}_z(\eta) \widehat{\phi}_k(z)  \widehat{\phi}_k(\xi - z)  \widehat{\kappa}(\xi - z) \left|\sum_{i=1}^n a_i  \widehat{\delta}_{x_i}(\xi) \widehat{\delta}_{s_i}(\eta) \right|^2.
\end{align*}
Since $\widehat{P}_z(\eta)$ and $\widehat{\kappa}(\xi - z)$ are nonnegative and $0 \leq \widehat{\phi}_1 \leq \widehat{\phi}_2$, we obtain \eqref{6.3}, and hence \eqref{6.1} as in the proof of Lemma \ref{lemma4.2}. Similarly, we can use the inequality 
\[0 \leq c_2^{-(1+ \frac{d}{2})}\, \widehat{R}_z(\eta) \leq \widehat{P}_z(\eta) \leq c_1^{-(1+ \frac{d}{2})}\, \widehat{Q}_z(\eta)\]
to prove \eqref{6.2}.
\end{proof}

\begin{proposition}\label{thm6.1}
Let $E \subset (0, \infty)^N$ and $F \subset \mathbb{R}^d$ be compact sets.
Suppose $F$ has Lebesgue measure zero and $d - \alpha n > 0$. Then 
$\P(W(E) \cap F \cap X(\mathbb{R}^n_+\setminus\{0\}) \neq \varnothing) > 0$ if and only if $C_{d - \alpha n}(E \times F) > 0$.
\end{proposition}

\begin{proof}
Define $p_{s, t}$ as in \eqref{E:pst}.
Let $g_0 = \delta_0$.
For $t \in \mathbb{R}^n_+ \setminus \{0\}$, let $g_t$ be the density of $X(t)$, namely,
\begin{align}
g_t(x) = (2\pi)^{-d} \int_{\R^d} e^{-i\langle x, \xi \rangle} e^{-\|t\|_1 \|\xi\|_2^\alpha /2} \, d\xi, \quad x \in \mathbb{R}^d,
\end{align}
where $\|t\|_1= \sum_{j=1}^n |t_j|$ for $t \in \mathbb{R}^n$. We extend the definition of $g_t$ for $t \in \mathbb{R}^n \setminus \{0\}$ by setting
\begin{align}\label{g:scaling}
g_t(x) = \|t\|_1^{-d/\alpha} g_{(1, \dots, 1)}(\|t\|_1^{-1/\alpha} x), \quad x \in \mathbb{R}^d.
\end{align}
Since $E \subset (0, \infty)^N$ and $F \subset \mathbb{R}^d$ are compact, we can find constants $0 < a < 1$ and $1 < b < \infty$ 
such that $s \in [a, b]^N$ for all $s \in E$, and $\|x\| < b$ for all $x \in F$.

\medskip

$(\Leftarrow)$. Suppose $C_{d - \alpha n}(E \times F) > 0$. Then $\mathscr{E}_{d - \alpha n}(\mu) < \infty$ for some $\mu\in P_0(E\times F)$.
Let $\varepsilon \in (0, 1)$, $f_{\varepsilon} = |B(0, \varepsilon)|^{-1} {\bf 1}_{B(0, \varepsilon)}$, $\phi_\varepsilon = f_{\varepsilon} \ast f_{\varepsilon}$ and
\begin{align}
 Z_{\varepsilon}(\mu) &= \int_{[1,2]^n} du \int_{E\times F} \phi_\varepsilon(W(s) - x) \, \phi_{\varepsilon}(X(u) - x) \,\mu(ds\,dx).
\end{align}
By the Cauchy--Schwarz inequality,
\begin{align}
\label{6.9}
\P(Z_\varepsilon(\mu) > 0) \geq \frac{ \E[Z_\varepsilon(\mu)]^2}{\E[Z_\varepsilon(\mu)^2]}.
\end{align}
Note that
\begin{align}
\begin{aligned}
\label{6.10}
\E[Z_\varepsilon(\mu)]
& = \int_{[1,2]^{n}} du \int \mu(ds\,dx) \, \E[\phi_\varepsilon(W(s) - x)] \E[\phi_\varepsilon(X(u) - x)]\\
& = \int_{[1,2]^n} du \int \mu(ds\,dx) \left[ \int_{\|w\| \leq b + 2} \phi_\varepsilon(w - x) \frac{1}{(2\pi \prod_{i=1}^N s_i)^{d/2}}\, e^{-\frac{\|w\|^2}{2\prod_{i=1}^N s_i}} 
\,dw \right]\\
& \hspace{105pt} \times \left[\int_{\|z\| \leq b + 2} \phi_\varepsilon(z - x) g_u(z) \, dz \right]\\
& \geq \int_{[1,2]^n} du \int \mu(ds\,dx) \left[ (2\pi b^N)^{-d/2}\,e^{-\frac{(b+2)^2}{2a^N}} \int_{\|w\|\le b+2}  \phi_\varepsilon(w - x) \, dw \right] \\
& \hspace{105pt} \times
\left[ \inf_{u \in [1,2]^n, \|v\| \leq b + 2}g_u(v) \int_{\|z\| \le b+2} \phi_\varepsilon(z - x) \, dz\right]\\
& = c_{\ref{6.10}} > 0.
\end{aligned}
\end{align}
To obtain the last line, we have used the fact that
\[
	\inf_{u \in [1,2]^n, \|v\| \leq b + 2}g_u(v) > 0,
\]
which follows from the scaling property \eqref{g:scaling}, the property that $g_{(1,\dots,1)}(a) > 0$ for all $a \in \R^n$ (see \cite[Corollary 3.2.1, p.379]{K}) and continuity.
Next, we establish an upper bound for
\begin{align}
\begin{aligned}\label{6.11}
\quad \E[Z_\varepsilon(\mu)^2]
& = \iint_{[1,2]^{2n}} du\, dv \iint \E[\phi_\varepsilon(W(s) - x) \phi_\varepsilon(W(t) - y)] \\
& \hspace{85pt} \times \E[\phi_\varepsilon(X(u) - x) \phi_\varepsilon(X(v) - y)] \,\mu(ds\,dx)\,\mu(dt\,dy).
\end{aligned}
\end{align}
Let $Z_1 = X(u) - X(u \wedge v)$ and $Z_2 = X(v) - X(u \wedge v)$, where $u \wedge v = (u_1 \wedge v_1, \dots, u_n \wedge v_n)$. 
Then, by \eqref{E:phi_eps:bd} and triangle inequality,
\begin{align}
\begin{aligned}
\label{6.12}
& \quad \E[\phi_\varepsilon(X(u) - x) \phi_\varepsilon(X(v) - y)]\\
& \leq 4^d \E[f_{2\varepsilon}(X(u) - x) f_{2\varepsilon}(X(v) - y)]\\
& \leq c_{\ref{6.12}}\,\varepsilon^{-d} \E[{\bf 1}_{B(0, 2\varepsilon)}(X(u) - x) f_{4\varepsilon}(X(v) - X(u) +x - y)]\\
& = c_{\ref{6.12}} \, \varepsilon^{-d} \E[P(|Z_1 + X(u\wedge v) - x| \leq 2 \varepsilon |Z_1, Z_2) f_{4\varepsilon}(Z_2 - Z_1 + x - y)].
\end{aligned}
\end{align}
Since $X(u \wedge v)$ is unimodal (see \cite[Remark 2.3]{KX2002}) and $Z_1, Z_2$, $X(u \wedge v)$ are independent, we have
\begin{align}
\label{6.13}
\P(|Z_1 + X(u\wedge v) - x| \leq 2 \varepsilon |Z_1, Z_2) \leq \P(|X(u \wedge v)| \leq 2 \varepsilon) \leq c_{\ref{6.13}}\, \varepsilon^d.
\end{align}
The preceding together with \eqref{E:phi_eps:bd} implies that
\begin{align}
\label{6.14}
\E[\phi_\varepsilon(X(u) - x) \phi_\varepsilon(X(v) - y)] \leq c_{\ref{6.14}} (\phi_{8\varepsilon} \ast g_{u - v})(x - y).
\end{align}
From \eqref{7.9} -- \eqref{5.17} of the proof of Theorem \ref{T:hit}, we have
\begin{align}
\label{6.15}
\E[\phi_{\varepsilon}(W(s) - x) \phi_{\varepsilon}(W(t) - y)] \leq c_{\ref{6.15}} (\phi_{8\varepsilon} \ast r_{s, t})(x - y),
\end{align}
where
\begin{align}
\label{6.16}
r_{s, t}(w) = \frac{1}{\|t - s\|^{d/2}} \exp\left(-\frac{\|w\|^2}{2c_{\ref{6.16}}\|t - s\|}\right), \quad w \in \mathbb{R}^d,
\end{align}
and $c_{\ref{6.16}} > 1$ is a constant.
Applying \eqref{6.14} and \eqref{6.15} to \eqref{6.11} yields
\begin{align}
\begin{aligned}
\label{6.17}
& \E[Z_\varepsilon(\mu)^2]\\
& \leq c_{\ref{6.17}} \iint_{[1,2]^{2n}} du\, dv\iint (\phi_{8\varepsilon} \ast r_{s, t})(x - y)\, 
(\phi_{8\varepsilon} \ast g_{u - v})(x - y) \,\mu(ds\,dx)\, \mu(dt\,dy).
\end{aligned}
\end{align}
By the change of variables $u' = u$, $v'= u - v$ and the symmetry of $g$, we have
\begin{align}
\begin{aligned}
\label{6.18}
\int_{[1,2]^n} du \int_{[1,2]^n} dv \, g_{u - v}(z) & = \int_{[1,2]^n} du' \int_{[1,2]^n - u'} dv' \, g_{v'}(z) \\
& \leq \int_{[1,2]^n} du' \int_{[-1,1]^n} dv' \, g_{v'}(z)\\
& = 2^n \int_{[0,1]^n} g_u(z) \, du := 2^n \,\kappa(z).
\end{aligned}
\end{align}
Thus,
\begin{align}
\label{6.19}
\E[Z_\varepsilon(\mu)^2] \leq c_{\ref{6.19}} \iint (\phi_{8\varepsilon} \ast r_{s, t})(x - y)\, (\phi_{8\varepsilon} \ast \kappa)(x - y) \,
\mu(ds\,dx)\, \mu(dt\,dy).
\end{align}
Note that $\kappa$ is integrable, positive definite and lower semicontinuous, and $\phi_{8\varepsilon} \ast \kappa$ is continuous,
 so we can apply Lemma \ref{lemma6.3} to get
\begin{align} 
\begin{aligned}\label{6.20}
\E[Z_\varepsilon(\mu)^2] & \leq c_{\ref{6.20}} \iint (\phi_{8\varepsilon} \ast p_{s, t})(x - y) \, (\phi_{8\varepsilon} \ast \kappa)(x - y)\, 
\mu(ds\,dx)\, \mu(dt\,dy).
\end{aligned}
\end{align}
Putting \eqref{6.9}, \eqref{6.10} and \eqref{6.20} together, we have
\begin{align}
\begin{aligned} 
\label{6.21}
& \P(W(E) \cap F^{8\varepsilon} \cap X([1,2]^n) \neq \varnothing) \geq \P(Z_\varepsilon(\mu) > 0)\\
& \geq c_{\ref{6.21}} \left[\iint (\phi_{8\varepsilon} \ast p_{s, t})(x - y) \, (\phi_{8\varepsilon} \ast \kappa)(x - y)\,
\mu(ds\,dx)\, \mu(dt\,dy)\right]^{-1}.
\end{aligned}
\end{align}
Let $(\varepsilon_m)_{m=1}^\infty$ be a sequence in $(0, 1)$ such that $\varepsilon_m \downarrow 0$.
Since $\{ W(E) \cap F \cap \overline{X([1,2]^n)} \neq \varnothing \}$ contains the intersection $\bigcap_{m=1}^\infty \{ W(E) \cap F^{8\varepsilon_m} \cap X([1,2]^n) \neq \varnothing \}$, we have
\begin{align}
\begin{aligned} 
\label{6.22}
&\P(W(E) \cap F \cap \overline{X([1,2]^n)} \neq \varnothing) \\
& \geq \lim_{m \to \infty} \P(W(E) \cap F^{8\varepsilon_m} \cap X([1,2]^n) \neq \varnothing) \\
& \geq c_{\ref{6.21}} \left[ \iint p_{s, t}(x - y) \, \kappa(x - y) \, \mu(ds\, dx) \, \mu(dt\, dy) \right]^{-1}\\
& \geq c_{\ref{6.22}} \,\mathscr{E}_{d - \alpha n}(\mu)^{-1}.
\end{aligned}
\end{align}
The last inequality follows from the fact that there are constants $c_{\ref{6.23}}$, $C_{\ref{6.23}} > 0$ such that
\begin{align}\label{6.23}
\frac{c_{\ref{6.23}}}{\|z\|^{d - \alpha n}} \leq \kappa(z) \leq \frac{C_{\ref{6.23}}}{\|z\|^{d - \alpha n}}
\end{align}
for all $z \in \mathbb{R}^d$ with $\|z\| \leq 2b$ (see \cite[Proposition 4.6]{KX2015}).
Since $\mathscr{E}_{d - \alpha n}(\mu) < \infty$, we have
\begin{align}
\P(W(E) \cap F \cap \overline{X([1,2]^n)} \neq \varnothing) > 0. 
\end{align}
From the proof of Lemma 4.1 in \cite{KX2005}, this implies that the set $(W(E) \cap F) \ominus \overline{X([1,2]^n)}$ has 
positive Lebesgue measure with positive probability,
where $A \ominus B$ denotes the set $\{ a - b : a \in A, b \in B\}$.
Since $X$ has only countably many jumps, the same is true for $(W(E) \cap F) \ominus X([1,2]^n)$.
Then by Lemma 4.1 in \cite{KX2005}, the set $W(E) \cap F \cap X([1,2]^n)$ is nonempty with positive probability, hence 
so is $W(E) \cap F \cap X(\mathbb{R}^n_+\setminus \{0\})$.

\medskip

($\Rightarrow$). Suppose $\P(W(E) \cap F \cap X(\mathbb{R}^n_+\setminus \{0\}) \neq \varnothing) > 0$. 
Then we can find rational numbers $0 < m < M$ such that $\P(W(E) \cap F \cap X(G) \neq \varnothing) > 0$, where 
$G = [0, M]^n\setminus [0, m)^n$.
Fix a point $\Delta \in \mathbb{R}^N_+ \setminus E$.
We define an $E \cup \{\Delta\}$-valued random variable $S$ as follows. Fix any $\omega \in \Omega$.

(i). If there is no $s \in E$ such that $W(s) \in F \cap X(G)$, define $S = \Delta$.

(ii). If there exists $s \in E$ such that $W(s) \in F \cap X(G)$, define $S = (S_1, \dots, S_N)$ inductively as follows:
\begin{align}
\begin{aligned}
S_1 &= \inf A_1= \inf\{ s_1 : \exists \, s_2, \dots, s_N, \exists \, x \in \overline{X(G)} \text{ such that }\\
& \hspace{150pt} (s_1, \dots, s_N) \in E, W(s_1, \dots, s_N) = x \in F\},\\
S_2 & = \inf A_2 = \inf\{ s_2 : \exists \, s_3, \dots, s_N, \exists \, x \in \overline{X(G)} \text{ such that }\\
& \hspace{120pt} (S_1, s_2, \dots, s_N) \in E, W(S_1, s_2, \dots, s_N) = x \in F\},\\
\vdots\\
S_N & = \inf A_N = \inf\{ s_N : \exists \, x \in \overline{X(G)} \text{ such that }\\
& \hspace{100pt} (S_1,  \dots, S_{N-1}, s_N) \in E, W(S_1 \dots, S_{N-1}, s_N) = x \in F\}.
\end{aligned}
\end{align}
\normalsize
Let us verify that $S_1, \dots, S_N$ are well-defined.
First, (ii) implies $A_1 \neq \varnothing$. In particular, there are sequences $(s^k)_{k=1}^\infty$ in $E$ and $(x^k)_{k=1}^\infty$ in $\overline{X(G)}$ such that $s_1^k \downarrow S_1$ and $W(s^k) = x^k \in F$.
Since $E$ is compact, we can find a subsequence $(s^{k_j})_{j=1}^\infty$ so that $s^k \to (S_1, s_2, \dots, s_N) \in E$ and $x^{k_j} = W(s^{k_j}) \to W(S_1, s_2, \dots, s_N) := x \in F$. Since $\overline{X(G)}$ is closed, $x$ is also in $\overline{X(G)}$ and this implies 
$A_2 \neq \varnothing$.
By induction, we can show that $A_m \neq \varnothing$ for all $m$, $S \in E$ and $W(S) \in F \cap \overline{X(G)}$.
Hence $S$ is well-defined.

Using the compactness of $E$, $F$ and $G$, continuity of $W$, and right-continuity of $X$, it can be shown that $S$ is measurable.
Since $\P(S \neq \Delta) = \P(W(E) \cap F \cap X(G) \neq \varnothing) > 0$ and $(S, W(S))$ is measurable, we can define a 
Borel probability measure $\mu$ on $E \times F$ by
\begin{align}
\mu(B) = \P((S, W(S)) \in B \mid S \neq \Delta)
\end{align}
for any Borel subset $B$ of $E \times F$.
Since $F$ has Lebesgue measure 0, $\mu(\{s\} \times F) = 0$ for all $s \in E$.

We aim to show that $\mathscr{E}_{d - \alpha n}(\mu) < \infty$.
For each $\varepsilon > 0$, let 
\begin{align}\label{6.34}
\phi_\varepsilon(x) &= \frac{1}{(2\pi\varepsilon^2)^{d/2}} \exp\left(-\frac{\|x\|^2}{2\varepsilon^2}\right), \quad x \in \mathbb{R}^d,\\
Z_\varepsilon(\mu) &= \int_{G'} dv \int_{E \times F} \mu(dt\, dy)\, \phi_\varepsilon(W(t) - y) \,\phi_\varepsilon(X(v) - y),
\end{align}
where $G' = [0, 2M]^n \setminus [0, M)^n$.
Fix any $\pi \subset \{1, \dots, N\}$.
For any $s \in \mathbb{R}^N_+$ and $u \in \mathbb{R}^n_+$, define $\mathscr{F}^\pi_s = \sigma\{W(t) : t \preccurlyeq_\pi s\}$ and $\mathscr{G}_u = \bigvee_{i=1}^n \mathscr{G}^{(i)}_{u_i}$, where $\mathscr{G}^{(i)}_{u_i} = \sigma\{X^{(i)}(t) : 0 \leq t \leq u_i\}$. 
Let $\mathscr{H}^\pi_{s, u} = \mathscr{F}^\pi_s \vee \mathscr{G}_u$. 
Fix $(s, u) \in (\mathbb{Q}_+ \cap [a, b])^N \times (\mathbb{Q}_+ \cap [0, M])^n$.
By the independence of $W$ and the $X^{(i)}$'s,
\begin{align}\label{6.32}
\notag
&\E[Z_\varepsilon(\mu) | \mathscr{H}^\pi_{s, u}] \\
& = \int_{G'} dv \int \mu(dt\,dy) \, \E[\phi_\varepsilon(W(t) - y)|\mathscr{F}^\pi_s] \, \E[\phi_\varepsilon(X(v) - y)|\mathscr{G}_u]\\
& \geq \int_{\{v \in G' :\, v \succcurlyeq u\}} dv \int_{\substack{t \succcurlyeq_\pi s \\ t \neq s}} \mu(dt\,dy) \, \E[\phi_\varepsilon(W(t) - y)|\mathscr{F}^\pi_s]\, \E[\phi_\varepsilon(X(v) - y)|\mathscr{G}_u]\, {\bf 1}_{\{\|W(s)\| < b\}}.
\notag
\end{align}
From \eqref{5.18} -- \eqref{5.22} in the proof of Theorem \ref{T:hit}, we see that on $\{\|W(s)\| < b\}$,
\begin{align} \label{6.33}
\E[\phi_\varepsilon(W(t) - y)|\mathscr{F}^\pi_s] \geq c_{\ref{6.33}}\, (\phi_\varepsilon \ast q_{s, t})(W(s) - y),
\end{align}
where
\begin{align}
\label{6.35}
q_{s, t}(w) = \frac{1}{\|t - s\|^{d/2}} \exp\left(-\frac{\|w\|^2}{2c_{\ref{6.35}}\|t - s\|} \right), \quad w \in \mathbb{R}^d
\end{align}
and $0 < c_{\ref{6.35}} < 1$ is a constant. Also, for all $v \succcurlyeq u$,
\begin{align}
\E[\phi_\varepsilon(X(v) - y)|\mathscr{G}_u] = (\phi_\varepsilon \ast g_{v - u})(X(u) - y).
\end{align}
It follows that
\begin{align}
\begin{aligned} \label{6.36}  
&\E[Z_\varepsilon(\mu)|\mathscr{H}_{s, u}^\pi]\\
& \geq c_{\ref{6.36}} \int_{\substack{t \succcurlyeq_\pi s \\ t \neq s}} (\phi_{\varepsilon} 
\ast q_{s, t})(W(s) - y)\, (\phi_\varepsilon \ast {\kappa})(X(u) - y) \, \mu(dt\,dy)\,{\bf 1}_{\{\|W(s)\| < b\}},
\end{aligned}
\end{align}
where
\begin{align}
\kappa(z) := \int_{[0, M]^n} g_v(z) dv \leq \int_{\{ v \in G' :\, v \succcurlyeq u\}} g_{v - u}(z) \, dv.
\end{align}
For any fixed $\omega \in \{S \neq \Delta\}$, since $W(S) \in \overline{X(G)}$, we can find a sequence $\{(s_m, u_m)\}_{m=1}^\infty$ 
in $(\mathbb{Q}_+ \cap [a, b])^N \times (\mathbb{Q}_+ \cap [0, M])^n$ such that $s_m \to S$ and $X(u_m) 
\to W(S)$. Since $W(S) \in F$, we have $\|W(S)\| < b$ and $\|W(s_m)\| < b$ for large $m$. 
Note that $(s, x) \mapsto (\phi_\varepsilon \ast q_{s, t})(x - y) (\phi_{\varepsilon} \ast \kappa)(x - y)$ is continuous.
Applying \eqref{6.36} and using Fatou's lemma, we have
\begin{align}
\notag
& \sup_{(s, u) \in \mathbb{Q}^N_+ \times \mathbb{Q}^n_+}  \E[Z_\varepsilon(\mu)| \mathscr{H}^\pi_{s, u}] \\
\notag
& \geq c_{\ref{6.36}} \liminf_{m\to\infty} \int_{\substack{t \succcurlyeq_\pi s_m\\t \neq s_m}} (\phi_\varepsilon \ast q_{s_m, t})(W(s_m) - y) \, (\phi_\varepsilon \ast \kappa)(X(u_m) - y) \, \mu(dt \, dy)\, {\bf 1}_{\{\|W(s_m)\| \leq b\}}\\
& \geq c_{\ref{6.36}} \int_{\substack{t \succcurlyeq_\pi S\\ t \neq S}} (\phi_{\varepsilon} \ast q_{S, t})(W(S) - y)\, (\phi_\varepsilon \ast {\kappa})(W(S) - y) \, \mu(dt\,dy) \, {\bf 1}_{\{S \neq \Delta\}}.
\end{align}
The above still holds for $\omega \notin \{S = \Delta\}$.
Taking square and expectation, and then using the Cauchy--Schwarz inequality, we get that
\begin{align}
\begin{aligned}\label{6.39}
& \E \Bigg( \sup_{(s, u) \in \mathbb{Q}^N_+ \times \mathbb{Q}^n_+}  \E[Z_\varepsilon(\mu)| \mathscr{H}^\pi_{s, u}]^2 \Bigg)\\
& \geq c_{\ref{6.36}}^2 P(S \neq \Delta) \int \left(\int_{t \succcurlyeq_\pi s} (\phi_{\varepsilon} \ast q_{s, t})(x - y)\, (\phi_\varepsilon 
\ast {\kappa})(x - y) \, \mu(dt\,dy) \right)^2 \mu(ds\, dx)\\
& \geq c_{\ref{6.36}}^2 \P(S \neq \Delta) \left( \iint_{t \succcurlyeq_\pi s} (\phi_{\varepsilon} \ast q_{s, t})(x - y)\, (\phi_\varepsilon 
\ast {\kappa})(x - y) \, \mu(dt\,dy) \, \mu(ds\, dx) \right)^2\\
& := c_{\ref{6.39}} \P(S \neq \Delta) J_{\varepsilon, \pi}^2.\\
\end{aligned}
\end{align}
Next, we find an upper bound for the expectation in \eqref{6.39}.
By Cairoli's maximal inequality (Lemma \ref{lemma6.2}), we have
\begin{align}
\begin{aligned}
& \E \Bigg( \sup_{(s, u) \in \mathbb{Q}^N_+ \times \mathbb{Q}^n_+} \E[Z_\varepsilon(\mu)| \mathscr{H}^\pi_{s, u}]^2 \Bigg)
 \leq 4^{N+n} \E[Z_\varepsilon(\mu)^2]\\
& = 4^{N+n} \iint_{G'\times G'} du\, dv \iint \mu(ds\,dx)\, \mu(dt\,dy)\, \E[\phi_\varepsilon(W(s) - x)\phi_\varepsilon(W(t) - y)]\\
& \hspace{195pt} \times \E[\phi_\varepsilon(X(u) - x)\phi_\varepsilon(X(v) - y)].
\end{aligned}
\end{align}
From \eqref{7.37} -- \eqref{7.50} in the proof of Theorem \ref{T:hit}, we see that for $(s, x), (t, y) \in E \times F$,
\begin{align}
\begin{aligned} \label{6.41}
\E[\phi_\varepsilon(W(s) - x)\phi_\varepsilon(W(t) - y)]
& \leq c_{\ref{6.41}} \E[\phi_{2\varepsilon}(W(s) - x)\phi_{2\varepsilon}(W_s(t) + C_{s, t} x - y)]\\
& \leq \tilde{c}_{\ref{6.41}} (\phi_{2\varepsilon} \ast r_{s, t})(x - y),
\end{aligned}
\end{align}
where
\begin{align}\label{6.42}
r_{s, t}(w) = \frac{1}{\|t - s\|^{d/2}} \exp\left(-\frac{\|w\|^2}{2c_{\ref{6.42}}\|t - s\|} \right), \quad w \in \mathbb{R}^d,
\end{align}
and $c_{\ref{6.42}} > 1$ is a constant. By a similar argument, we can show that for $u, v \in G'$,
\begin{align}
\begin{aligned}\label{6.43}
\E[\phi_\varepsilon(X(u) - x)\phi_\varepsilon(X(v) - y)] \leq c_{\ref{6.43}} \E[\phi_{2\varepsilon}(X(u) - x) 
\phi_{2\varepsilon}(X(v) - X(u) + x - y)].
\end{aligned}
\end{align}
Write $Z_1 = X(u) - X(u \wedge v)$ and $Z_2 = X(v) - X(u \wedge v)$. Then
\begin{align}
\begin{aligned}\label{E:4.40}
& \quad \E[\phi_\varepsilon(X(u) - x)\phi_\varepsilon(X(v) - y)]\\
& \leq c_{\ref{6.43}} \E[ \E(\phi_{2\varepsilon}(Z_1 + X(u \wedge v) - x)|Z_1, Z_2)\,\cdot  \phi_{2\varepsilon}(Z_2 - Z_1 + x - y)]\\
& = c_{\ref{6.43}} \E\left[ \int \phi_{2\varepsilon}(Z_1 + z - x) \, g_{u \wedge v}(z) \, dz\, \cdot \phi_{2\varepsilon}(Z_2 - Z_1 + x - y)\right]\\
& \leq c_{\ref{6.43}} \E\left[ \sup_{u, v \in G'}g_{u \wedge v}(0) \int \phi_{2\varepsilon}(Z_1 + z - x) \, dz \cdot \phi_{2\varepsilon}(Z_2 - Z_1 
+ x - y) \right]\\
& \leq {c}_{\ref{E:4.40}} \E[\phi_{2\varepsilon}(Z_2 - Z_1 + x - y)]\\
& = {c}_{\ref{E:4.40}} (\phi_{2\varepsilon} \ast g_{v - u})(x - y).
\end{aligned}
\end{align}
Note that $\displaystyle{\sup_{u, v \in G'} g_{u \wedge v}(0)}$ is a positive finite constant, thus so is $\tilde{c}_{\ref{6.43}}$.
Now, we have
\begin{align}
\begin{aligned}\label{6.45}
& \E \Bigg( \sup_{(s, u) \in \mathbb{Q}^N_+ \times \mathbb{Q}^n_+} \E[Z_\varepsilon(\mu)| \mathscr{H}^\pi_{s, u}]^2 \Bigg)\\
& \leq c_{\ref{6.45}} \iint_{G' \times G'} du\, dv \iint (\phi_{2\varepsilon} \ast r_{s, t})(x - y) \, (\phi_{2\varepsilon} \ast g_{v - u})(x - y) 
\, \mu(ds\,dx) \, \mu(dt\,dy).
\end{aligned}
\end{align}
By a change of variables and the symmetry of $g$,
\begin{align}
\begin{aligned}
\iint_{G' \times G'} g_{v - u}(z) \, du \, dv & \leq \int_{[0, 2M]^n} du \int_{[0, 2M]^n} dv\,  g_{v - u}(z)\\
& = \int_{[0, 2M]^n} du' \int_{[0, 2M]^n - u'} dv' \, g_{v'}(z)\\
& \leq (2M)^n \int_{[-2M, 2M]^n} g_{v'}(z) \, dv'\\
& = (2M)^n 2^n \int_{[0, 2M]^n} g_u(z) \, du.\\
\end{aligned}
\end{align}
Then by Lemma 4.7 of \cite{KX2015}, the above is
\begin{align}\label{6.47}
\leq c_{\ref{6.47}} \int_{[0, M]^n} g_u(z) \, du = c_{\ref{6.47}}\, \kappa(z)
\end{align}
for some constant $c_{\ref{6.47}} > 0$.
Using this and Lemma \ref{lemma6.3}, we get
\begin{align}
\begin{aligned}\label{6.48}
& \E \Bigg( \sup_{(s, u) \in \mathbb{Q}^N_+ \times \mathbb{Q}^n_+}  \E[Z_\varepsilon(\mu)| \mathscr{H}^\pi_{s, u}]^2 \Bigg)\\
& \leq c_{\ref{6.45}}\, c_{\ref{6.47}} \iint (\phi_{2\varepsilon} \ast r_{s, t})(x - y) \, (\phi_{2\varepsilon} \ast {\kappa})(x - y) \, 
\mu(ds\,dx) \, \mu(dt\,dy)\\
& \leq c_{\ref{6.48}} \iint (\phi_{\varepsilon} \ast p_{s, t})(x - y) \, (\phi_{\varepsilon} \ast {\kappa})(x - y) \, \mu(ds\,dx) \, \mu(dt\,dy) 
:= c_{\ref{6.48}} \, I_\varepsilon.
\end{aligned}
\end{align}
Combining \eqref{6.39} and \eqref{6.48}, and then summing over all $\pi \subset \{1, \dots, N\}$ yields
\begin{align}
\begin{aligned}\label{6.49}
&2^N c_{\ref{6.48}} \, I_\varepsilon 
\geq c_{\ref{6.39}} \,\P(S \neq \Delta) \sum_\pi J_{\varepsilon, \pi}^2\\
& \geq 2^{-N} c_{\ref{6.39}}\, \P(S \neq \Delta) \left( \sum_{\pi} J_{\varepsilon, \pi} \right)^2\\
& = 2^{-N} c_{\ref{6.39}} \, \P(S \neq \Delta) \left( \iint (\phi_\varepsilon \ast q_{s, t})(x-y) \, (\phi_\varepsilon \ast \kappa)(x-y) \,
 \mu(ds\, dx) \, \mu(dt\, dy) \right)^2\\
& \geq c_{\ref{6.49}} \, \P(S \neq \Delta) \, I_\varepsilon^2,
\end{aligned}
\end{align}
where Lemma \ref{lemma6.3} is used in the last inequality.
So we have
\begin{align}
\begin{aligned}\label{6.50}
c_{\ref{6.50}}\, I_\varepsilon \geq \P(S \neq \Delta) I_\varepsilon^2
\end{aligned}
\end{align}
for some constant $c_{\ref{6.50}} > 0$ independent of $\varepsilon$.
Note that $I_\varepsilon$ is finite for each $\varepsilon > 0$, so
\begin{align}
 I_\varepsilon = \iint (\phi_{\varepsilon} \ast p_{s, t})(x - y) \, (\phi_{\varepsilon} \ast {\kappa})(x - y) \, \mu(ds\,dx) \, \mu(dt\,dy) 
 \leq c_{\ref{6.50}} \, \P(S \neq \Delta)^{-1}.
\end{align}
Letting $\varepsilon \to 0$ gives
\begin{align}
\iint p_{s, t}(x - y) \,\kappa(x - y) \, \mu(ds\,dx) \, \mu(dt\,dy)< \infty.
\end{align}
By Proposition 4.6 of \cite{KX2015}, we have $\mathscr{E}_{d - \alpha n}(\mu) < \infty$.
The proof is complete.
\end{proof}

\section{Hausdorff dimension of $W(E)\cap F$ and proofs of Theorems \ref{T:dim:cap} and \ref{T:dim:dim-d}}

In this section, we study the essential supremum of the Hausdorff dimension of $W(E)\cap F$ and prove Theorems \ref{T:dim:cap} and \ref{T:dim:dim-d}.

\begin{proposition}\label{prop4.5}
If $E \subset (0, \infty)^N$, $F \subset \mathbb{R}^d$ are compact sets and $F$ has positive Lebesgue measure, then
\[ 
\|\dimh (W(E) \cap F)\|_\infty = \min\{d, 2\dimh  E \}. 
\]
\end{proposition}

\begin{proof}
The proof is essentially the same as that of Proposition 1.2 of \cite{KX2015} and is thus omitted.
\end{proof}

\begin{proof}[Proof of Theorem \ref{T:dim:cap}]
First, suppose that $F$ has Lebesgue measure zero.
Let 
\[
\gamma^* = \sup\{ \gamma > 0 : C_\gamma(E \times F) > 0 \}.
\]
Proposition \ref{thm2.3} implies that $\gamma^* \le d$.
For any $\gamma \in (0, \gamma^*)$, we have $C_\gamma(E \times F) > 0$
and we can choose an integer $n \ge 1$ and $\alpha 
\in (0, 2)$ such that $d - \alpha n = \gamma$. If $X$ is an $n$-parameter, $d$-dimensional additive $\alpha$-stable process independent 
of $W$, then by Proposition \ref{thm6.1},
\[
\P(W(E) \cap F \cap X(\mathbb{R}^n_+\setminus\{0\}) \neq \varnothing) > 0.
\] 
This implies that
$\P(W(E) \cap F \cap X(\mathbb{R}^n_+\setminus\{0\}) \neq \varnothing \mid W) > 0$ on an event $A$ with positive probability. Then 
by Proposition \ref{prop4.1}, we have ${\EuScript Cap}_{d - \alpha n} \big(W(E) \cap F \big) > 0$ on the event $A$, and thus $\dimh \big(W(E) \cap F \big) 
\ge d - \alpha n = \gamma$ on the event $A$. It follows that $\|\dimh (W(E) \cap F)\|_\infty \ge \gamma$ for all $\gamma \in (0, \gamma^*)
$, hence  we have $\|\dimh (W(E) \cap F)\|_\infty \ge \gamma^*$. 

Next, we aim to show the reverse inequality.
This is obviously true if $\gamma^* = d$.
Let us now suppose that $\gamma^* < d$.
In this case, consider any $\gamma \in (\gamma^*, d)$, for which we have $C_\gamma(E \times F) = 0$. We can choose an integer 
$n \ge 1$ and $\alpha \in (0, 2)$ such that $d - \alpha n = \gamma$ and consider the corresponding additive stable process $X$. By Proposition \ref{thm6.1}, 
\[
\P(W(E) \cap F \cap X(\mathbb{R}^n_+\setminus \{0\}) \neq \varnothing) = 0.
\]
Then, Proposition \ref{prop4.1} shows that ${\EuScript Cap}_{d - \alpha n} \big(W(E) \cap F \big) = 0$ and thus $\dimh (W(E) \cap F) \le \gamma$ almost surely. 
This is true for arbitrary $\gamma > \gamma^*$.
Hence, $\|\dimh (W(E) \cap F)\|_\infty \leq \gamma^*$.

Now suppose that $F$ has a positive Lebesgue measure.
According to Proposition \ref{prop4.5}, we have
\[ \|\dimh (W(E) \cap F)\|_\infty = \min\{d, 2\dimh  E\}. \]
Note that $\dimh  F = \dim_{_{\rm P}} F = d$ because $F$ has positive Lebesgue measure. 
Then by Proposition \ref{thm2.3}, we have $\sup\{ \gamma > 0 : C_\gamma(E \times F) > 0 \} = \min\{d, 2\dimh  E\}$. The result follows.
\end{proof}

\begin{proof}[Proof of Theorem \ref{T:dim:dim-d}]
The proof is similar to the proof of Theorem 1.1 in \cite{KX2015}. Instead of Kaufman's uniform dimension result for the Brownian motion, 
we use the uniform dimension result for the Brownian sheet which holds for $d \ge 2N$ (see \cite{Mountford}, \cite{KWX06}). 
Other parts of the proof are almost the same. The only part that requires additional proof is the uniform estimates for hitting probabilities of the Brownian 
 sheet and additive stable process.
More precisely, suppose that $E \subset [1/q, q]^N$ and $F \subset [-q, q]^d$ where $1 < q < \infty$.
For each $r \in (0, 1)$, let $T(r)$ be the collection of open balls of radius $r^2$ contained in $[1/q, q]^N$, and $S(r)$ be the collection of open balls of radius $r$ in $[-q, q]^d$.
We need to show that there is a constant $0 < c_{\ref{9.19}} < \infty$ such that for all $r \in (0, 1)$,
\begin{align} \label{9.19}
\sup_{I \in T(r)} \sup_{J \in S(r)}  \P(W(I) \cap J \ne \varnothing) \leq c_{\ref{9.19}} \, r^{d},
\end{align}
and for an $n$-parameter, $\mathbb{R}^N$-valued additive $\alpha$-stable process $X(u) = X^{(1)}(u_1) + \dots + X^{(n)}(u_n)$ and 
$G= [0, A)^n \setminus [0, a)^n$, $0 < a < A < \infty$, there is $0 < c_{\ref{9.28}} < \infty$ such that for all $r \in (0, 1)$,
\begin{align} \label{9.28}
\sup_{I \in T(r)} \P(X(G) \cap I \ne \varnothing) \leq c_{\ref{9.28}} \, r^{2(N- \alpha n)}.
\end{align}

To prove \eqref{9.19}, take $I = B(t_0, r^2)$ and $J = B(x_0, r)$. For each $r \in (0, 1)$, let
\begin{align}
Z_r = \int_{B(t_0, r^2)} {\bf 1}_{B(x_0, 2r)}(W(t)) \, dt.
\end{align}
Fix any $s \in I$. Since $W_s(t)$ and $\mathscr{F}_s^\pi$ are independent for $t \succcurlyeq_\pi s$ and $\pi \subset \{1, \dots, N\}$, 
we have
\begin{align}\begin{aligned}\label{9.21}
&\sum_\pi \E(Z_r|\mathscr{F}^\pi_s)\\
& \ge \sum_\pi \int_{\substack{t \in B(t_0, r^2)\\ t \succcurlyeq_\pi s \quad}} \P(\|W_s(t) + C_{s, t} W(s) - x_0\| \leq 2r| \mathscr{F}_s^\pi )
\, dt \, {\bf 1}_{\{\|W(s) - x_0\| \leq r\}}\\
& \ge\sum_\pi \int_{\substack{t \in B(t_0, r^2)\\ t \succcurlyeq_\pi s \quad}} \P(\|W_s(t) - (1 - C_{s, t})x_0\| \leq r| \mathscr{F}_s^\pi )\, dt \, 
{\bf 1}_{\{\|W(s) - x_0\| \leq r\}}\\
& = \sum_\pi \int_{\substack{t \in B(t_0, r^2)\\ t \succcurlyeq_\pi s \quad}} \P(\|W_s(t) - (1 - C_{s, t})x_0\| \leq r)\, dt \, {\bf 1}_{\{\|W(s) - x_0\| 
\leq r\}}\\
& = \int_{B(t_0, r^2)\setminus\{s\}} \P(\|W_s(t) - (1 - C_{s, t})x_0\| \leq r)\, dt \, {\bf 1}_{\{\|W(s) - x_0\| \leq r\}}.
\end{aligned}\end{align}
Then by \eqref{5.9}, we further deduce that
\begin{align}\label{9.22}
\notag
&\sum_\pi \E(Z_r|\mathscr{F}^\pi_s)\\
\notag
& \ge c \int_{B(t_0, r^2)\setminus\{s\}}dt \int_{B(0, r)} dw \, \frac{1}{\|s - t\|^{d/2}}\, \exp\left(-\frac{\|w - (1 - C_{s, t})x_0\|^2}
{2 c_{\ref{5.9}} \|s - t\|} \right)\, {\bf 1}_{\{\|W(s) - x_0\| \leq r\}}\\
\notag
& \ge c \int_{B(t_0, r^2)\setminus\{s\}} dt \int_{B(0, 1)} \, \frac{r^{d} \,dw }{\|s - t\|^{d/2}}\, \exp\left(-\frac{3r^2 \|w\|^2 + 3 c_{\ref{5.8}}^2 
q^d \|s - t\|}{2 c_{\ref{5.9}} \|s - t\|} \right)\, {\bf 1}_{\{\|W(s) - x_0\| \leq r\}}\\
\notag
& \ge c' \int_{B(t_0, r^2)\setminus \{s\}} dt \int_{B(0, 1)} \frac{r^d dw}{(2r^2)^{d/2}} \, \exp\left(-\frac{3r^2\|w\|^2}
{2c_{\ref{5.9}}\cdot 2r^2} \right)\, {\bf 1}_{\{\|W(s) - x_0\| \leq r\}}\\
& \ge  c_{\ref{9.22}}\, r^{2N}\, {\bf 1}_{\{\|W(s) - x_0\| \leq r\}}.
\end{align}
\normalsize
Let $\{q_1, q_2, \dots \}$ be an enumeration of $\mathbb{Q}_+$, and let $I_m = I \cap \{q_1, \dots, q_m\}^N$.
Note that by the continuity of $W(t)$, $\{W(I) \cap J \ne \varnothing\} = \bigcup_{m=1}^\infty \{W(I_m) \cap J \ne \varnothing\}$. 
Since $W(I_m) \cap J \ne \varnothing$ implies $\|W(s) - x_0\| \le r$ for some $s \in I_m$, it follows from \eqref{9.22} that
\begin{align}\begin{aligned}
c_{\ref{9.22}} \, r^{2N} \, {\bf 1}_{\{W(I_m) \cap J \ne \varnothing\}} 
&\le \sup_{s \in \mathbb{Q}_+^N} \sum_\pi \E(Z_r|\mathscr{F}_s^\pi) \\
&\le \sum_\pi \sup_{s \in \mathbb{Q}_+^N} \E(Z_r|\mathscr{F}_s^\pi).
\end{aligned}\end{align}
Taking square and using the Cauchy--Schwarz inequality, we have
\begin{align}\begin{aligned}
(c_{\ref{9.22}} \, r^{2N})^2 \,{\bf 1}_{\{W(I_m) \cap J \ne \varnothing\}} 
& \le \left( \sum_\pi \sup_{s\in \Q^N_+} \E(Z_r|\mathscr{F}_s^\pi) \right)^2\\
&\le 2^N \sum_\pi \sup_{s \in \mathbb{Q}_+^N} 
\E(Z_r|\mathscr{F}_s^\pi)^2.
\end{aligned}\end{align}
Taking expectations on both sides yields
\begin{align}
c_{\ref{9.22}}^2 \, r^{4N} \, \P(W(I_m) \cap J \ne \varnothing) &\le 2^N \sum_\pi \E \left(\sup_{s \in \mathbb{Q}_+^N} 
\E(Z_r|\mathscr{F}_s^\pi)^2\right).
\end{align}
Since $\{\mathscr{F}^\pi_s\}$ is commuting with respect to $\preccurlyeq_\pi$, we can use Cairoli's maximal inequality to get
\begin{align}\begin{aligned}\label{9.26}
c_{\ref{9.22}}^2 \, r^{4N} \, \P(W(I_m) \cap J \ne \varnothing)
& \le 2^{4N} \E(Z_r^2) \\
&\le 2^{4N} \int_{B(t_0, r^2)}ds  \int_{B(t_0, r^2)} dt \, \P(\|W(t) - x_0\| \leq 2r)\\
&\le c_{\ref{9.26}} \, r^{4N + d}.
\end{aligned}\end{align}
Letting $m \to \infty$ yields \eqref{9.19}.

To prove \eqref{9.28}, let $I = B(t_0, r^2)$, $G' = [0, 2A)^n\setminus [0, a)^n$, $\mathscr{G}_u = \bigvee_{i=1}^n 
\sigma\{X^{(i)}(t) : t \le u_i\}$. Let
\begin{align}
Z_r = \int_{G'} {\bf 1}_{B(t_0, 2r^2)}(X(v)) \, dv
\end{align}
and
\begin{align}
\kappa(x) = \int_{[0, A]^n} g_{v}(x) \, dv, \quad x \in \mathbb{R}^N,
\end{align}
where $g_v$ is defined as in the proof of Proposition \ref{thm6.1}.
By Proposition 4.6 of \cite{KX2015}, there are constants $0 < c_{\ref{9.31}} < C_{\ref{9.31}} < \infty$ such that for all 
$x \in [-1, 1]^N$, 
\begin{align}\label{9.31}
\frac{c_{\ref{9.31}}}{\|x\|^{N - \alpha n}} \le \kappa(x) \le \frac{C_{\ref{9.31}}}{\|x\|^{N - \alpha n}}.
\end{align}
Then for any $u \in G$,
\begin{align} \begin{aligned}\label{9.32}
\E(Z_r|\mathscr{G}_u) &\ge \int_{G'} \P(\|X(v) - t_0\| \leq 2 r^2|\mathscr{G}_u) \, dv \, {\bf 1}_{\{\|X(u) - t_0\| \leq r^2\}}\\
& \ge \int_{v \in G', v \succcurlyeq u} \P(\|X(v) - X(u)\| \le r^2|\mathscr{G}_u) \, dv \, {\bf 1}_{\{\|X(u) - t_0\| \le r^2\}}\\
& = \int_{v \in G', v \succcurlyeq u} \int_{B(0, r^2)} g_{v - u}(x)\, dx\, dv \, {\bf 1}_{\{\|X(u) - t_0\|\le r^2\}}\\
& \ge \int_{B(0, r^2)} \kappa(x) \, dx \, {\bf 1}_{\{\|X(u) - t_0\| \le r^2\}}\\
& \ge \int_{B(0, r^2)} \frac{c_{\ref{9.31}}}{\|x\|^{N - \alpha n}} \, dx \, {\bf 1}_{\{\|X(u) - t_0\| \le r^2\}}\\
& \ge c_{\ref{9.32}}\, r^{2\alpha n} \, {\bf 1}_{\{\|X(u) - t_0\| \le r^2\}}.
\end{aligned}\end{align}
Let $G_m = G \cap \{q_1, \dots, q_m\}^n$.
Note that by the right continuity of $X(u)$, $\{X(G) \cap I \ne \varnothing\} = \bigcup_{m=1}^\infty \{X(G_m) \cap I \ne \varnothing\}$. 
Since $X(G_m) \cap I \ne \varnothing$ implies $\|X(u) - t_0\| \le r^2$ for some $u \in G_m$, it follows from \eqref{9.32} that
\begin{align}
c_{\ref{9.32}}^2\, r^{4\alpha n} \, {\bf 1}_{\{X(G_m) \cap I \ne \varnothing\}} \le \sup_{u \in \mathbb{Q}_+^n} \E(Z_r|\mathscr{G}_u)^2.
\end{align}
Taking expectation and using Cairoli's maximal inequality, we obtain
\begin{align}\begin{aligned}\label{9.33}
c_{\ref{9.32}}^2 \, r^{4\alpha n} \, \P(X(G_m) \cap I \ne \varnothing) & \le 4^n \E(Z_r^2).
\end{aligned}\end{align}
Let $Y_1 = X(u) - X(u \wedge v)$ and $Y_2 = X(v) - X(u \wedge v)$. By the independence of $X(u \wedge v), Y_1, Y_2$, and 
\eqref{9.31}, we have the following upper bound:
\begin{align}\begin{aligned}\label{9.35}
\E(Z_r^2) &= \iint_{G' \times G'} \P(\|X(u) - t_0\| \le 2r^2, \|X(v) - t_0\| \le 2r^2) \, du\, dv\\
& \le \iint_{G' \times G'} \P(\|X(u) - t_0\| \le 2r^2, \|X(v) - X(u)\| \le 4r^2) \, du\, dv\\
& \le \iint_{G' \times G'} \E\left[ \P(\|Y_1 - X(u \wedge v) - t_0\| \le 2r^2|Y_1, Y_2)\cdot {\bf 1}_{\{\|Y_2 - Y_1\| \le 4r^2\}}\right] \, du \, dv\\
& \le \iint_{G'\times G'} \P(\|X(u\wedge v)\| \le 2r^2)\, \P(\|X(v) - X(u)\| \le 4r^2) \, du \, dv\\
& \le \sup_{u, v \in G'} g_{u \wedge v}(0) \cdot |B(0, 2r^2)| \iint_{G'\times G'}\, du \, dv \int_{B(0, 4r^2)} g_{v - u}(x) \, dx\\
& \le c\, r^{2N} \int_{B(0, 4r^2)} \kappa(x) \, dx\\
& \le c\, r^{2N} \int_{B(0, 4r^2)} \frac{C_{\ref{9.31}}}{\|x\|^{N -\alpha n}} \, dx\\
& \le c_{\ref{9.35}} \, r^{2N + 2\alpha n}.
\end{aligned}\end{align}
Combining this with \eqref{9.33} and letting $m \to \infty$ yields \eqref{9.28}.
With \eqref{9.19} and \eqref{9.28} in place, we can follow the proof of \cite[Theorem 1.1]{KX2015} to complete the proof of Theorem \ref{T:dim:dim-d}.
\end{proof}

\noindent{\bf Acknowledgments}:
The research of Y. Xiao was partially supported by the NSF grant DMS-2153846.
C. Y. Lee was partially supported by a research startup fund of the Chinese University of Hong Kong, Shenzhen and the Shenzhen Peacock fund 2025TC0013.

\end{document}